\documentclass[eqno,12pt]{amsart}
\usepackage[top=1in, bottom=1in, left=1.2in, right=1.2in]{geometry}
\usepackage{amsmath}
\usepackage{amssymb}
\usepackage{amsthm}
\usepackage{verbatim}
\usepackage{graphicx}
\usepackage{graphics}
\usepackage{color}
\usepackage{enumerate}
\usepackage{mathrsfs}
\usepackage{booktabs}
\usepackage[toc,page]{appendix}
\usepackage{enumerate}
\usepackage{mathrsfs}
\usepackage{fancyhdr}
\usepackage{latexsym, amsfonts, amssymb, amsmath,  amsthm}
\usepackage{amsopn, amstext, amscd,pifont}
\usepackage[all]{xy}
\usepackage{color}
\usepackage{graphicx,hyperref}
\usepackage{listings}
\usepackage{setspace}
\usepackage{url}
\usepackage{paralist}

\usepackage{pdfsync}

\usepackage{graphicx,pgf}

\theoremstyle{plain}
\newtheorem{theorem}{Theorem}[section]
\newtheorem{proposition}[theorem]{Proposition}
\newtheorem{lemma}[theorem]{Lemma}
\newtheorem{corollary}[theorem]{Corollary}

\theoremstyle{definition}
\newtheorem{definition}[theorem]{Definition}

\newtheorem{remark}[theorem]{Remark}

\newtheorem{introtheorem}{Theorem}

\graphicspath{{images/}}
\newcommand{\C}{\mathbb{C}}

\renewcommand{\L}{\mathbb{L}}
\newcommand{\N}{\mathbb{Z}_{>0}}
\newcommand{\Q}{\mathbb{Q}}
\newcommand{\R}{\mathbb{R}}

\renewcommand{\P}{\mathbb{P}}

\def\cO{{\mathcal{O}}}

\def\cU{{\mathcal{U}}}

\def\z{\zeta}
\def\ol{\overline}

\DeclareMathOperator{\Rat}{Rat}
\DeclareMathOperator{\rat}{rat}

\def\poneberk{{\mathbf{P}^1}}

\def\nUn{{\mathcal{U}_n}}

\begin{document}
\title{Indeterminacy Loci of Iterate Maps in Moduli Space}
\author{Jan KIWI and Hongming NIE}
\thanks{Jan Kiwi was partially supported by CONICYT PIA ACT172001  and by ``Proyecto Fondecyt \#1160550''}
\address{Facultad de Matem\'aticas, Pontificia Universidad Cat\'olica de Chile.}
\email{jkiwi@mat.uc.cl}
\address{Einstein Institute of Mathematics, The Hebrew University of Jerusalem}
 \email{hongming.nie@mail.huji.ac.il}
\date{\today}
\maketitle

\begin{abstract}
 The moduli space $\rat_d$ of rational maps in one complex variable of degree $d \ge 2$ has a natural compactification by a projective variety
$\ol{\rat}_d$ provided by geometric invariant theory.
Given $n \ge 2$, the iteration map $\Phi_n : \rat_d \to \rat_{d^n}$, defined by $\Phi_n: [f] \mapsto [f^n]$,
extends to a rational map $\Phi_n : \ol{\rat}_d \dasharrow \ol{\rat}_{d^n}$.
We characterize the elements of $\ol{\rat}_d$ which lie in the indeterminacy locus of $\Phi_n$.
\end{abstract}

\tableofcontents

\section{Introduction}
The space of complex rational maps $\mathrm{Rat}_d$ of degree $d \ge 2$ admits a natural compactification by $\mathbb{P}^{2d+1}$. For each $n \ge 2$,
the iterate map
$\Psi_n:\mathrm{Rat}_d\to\mathrm{Rat}_{d^n}$ defined by sending $f$ to $f^n$ extends to a rational map  $\Psi_n: \mathbb{P}^{2d+1}\dasharrow\mathbb{P}^{2d^n+1}$.
According to DeMarco \cite[Theorem 0.2]{DeMarco05} the map $\Psi_n$ has an indeterminacy locus $I(d)$ independent of $n \ge 2$.

The group $\mathrm{PSL}_2(\mathbb{C})$ acts on the space
$\mathrm{Rat}_d$ by conjugacy. The induced quotient space
$\mathrm{rat}_d$ is the moduli space of degree $d$ rational
maps. Moduli space is a complex orbifold of dimension
$2d-2$. Geometric invariant theory (GIT) provides a compactification
$\overline{\mathrm{rat}}_d$ of the moduli space $\mathrm{rat}_d$ by
considering the action of $\mathrm{PSL}_2(\mathbb{C})$ on the
semistable loci $\mathrm{Rat}_d^{ss}\subset\mathbb{P}^{2d+1}$
\cite{Silverman98}. The iterate map $\Psi_n$ induces a regular map
$\Phi_n:\mathrm{rat}_d\to\mathrm{rat}_{d^n}$ that sends the conjugacy
class $[f]$ to $[f^n]$, see \cite[Proposition 4.1]{DeMarco07}.
However, $\Phi_n$ does not extend continuously to the compactification
$\overline{\mathrm{rat}}_d$ \cite[Theorem 10.1]{DeMarco07}. That is,
$\Phi_n:\overline{\mathrm{rat}}_d\dasharrow\overline{\mathrm{rat}}_{d^n}$ has
a non-trivial indeterminacy locus denoted $I(\Phi_n)$.

Our main result gives a complete description of the indeterminacy
locus $I(\Phi_n)$. Our work answers a question posed by DeMarco
in~\cite{DeMarco07}.

\medskip
We say that $f \in \mathrm{Rat}_d^{ss} \setminus I(d)$ is \emph{$n$-unstable} if $f^n \notin \mathrm{Rat}_{d^n}^{ss}$. The subset of  $\mathrm{Rat}_d^{ss}$ formed by the $n$-unstable maps is denoted by $\cU_n$.

\begin{introtheorem}\label{main}
For $d\ge 2$ and $n\ge 2$, let $\Phi_n:\overline{\mathrm{rat}}_d\dasharrow\overline{\mathrm{rat}}_{d^n}$ be the iterate map and denote by $I(\Phi_n)$ its indeterminacy locus.
For all $f \in \Rat^{ss}_d$, we have that $[f] \in I(\Phi_n)$ if and only if  $f \in I(d)\cup\cU_n$.
\end{introtheorem}

As discussed in Section~\ref{upper-bound},  in our setting the indeterminacy locus coincides with the points where $\Phi_n$ has no continuous extension (in the analytic topology). Thus,
to prove our main theorem, we
only need to establish that for $f\in \Rat_d^{ss}$, the map $\Phi_n$ has no continuous extension at $[f]$ if and only if $f\in I(d)\cup\cU_n$.

For quadratic rational maps, the indeterminacy locus $I(\Phi_n) \subset \overline{\mathrm{rat}}_2 $ was explicitly described by DeMarco in~\cite[Theorem 5.1]{DeMarco07}.
Theorem~\ref{main} is an easy consequence of this description in the case of quadratic maps. Hence, in this paper, we focus on the case that $d\ge 3$.
 In the same work, DeMarco~\cite[Lemma 4.2]{DeMarco07} also proved that if $[f] \in I(\Phi_n)$ then
$f \in I(d) \cup \cU_n$, for even degrees. 
{Then she asked for the veracity of the converse, see \cite[Question]{DeMarco07}. Our Theorem \ref{main} answers her question in the affirmative.}

One direction of Theorem~\ref{main}: if $[f] \in I(\Phi_n)$ then  $f \in I(d)\cup\cU_n$ is easily obtained by extending a result by DeMarco~ \cite[Lemma 4.2]{DeMarco07} to all degrees (see Proposition~\ref{semistable}). The reverse implication, if $f\in I(d)\cup\cU_n$ then $[f] \in I(\Phi_n)$ requires substantial work. 
When $f\in I(d)\cup\cU_n$ we have that 
$f$ is a ``degenerate rational map''.
A degenerate rational map $f$
is defined by a pair of polynomials with shared zeros called the \emph{holes of $f$}.
To show that the map $\Phi_n$ is indeterminate at certain $[f] \in \ol{\rat}_d$, we construct holomorphic families $f_t$ and $g_t$ of (possibly degenerate) degree $d$ rational maps parametrized by  a neighborhood of $t=0$ in $\C$. These families are carefully chosen to
materialize an indeterminacy of $\Phi_n$ at $[f]$. More precisely, the constructions are such that both $[f_t]$ and $[g_t]$ converge to $[f]$ in $\ol{\rat}_d$ while the iterates $[f^n_t]$ and $[g^n_t]$ converge to different elements of
 $\ol{\rat}_{d^n}$. In almost all the cases, it is useful to  employ techniques from non-Archimedean rational dynamics.
Namely, the holomorphic families $f_t$ and $g_t$ act on the Berkovich projective line $\mathbf{P}^1$ over a suitable non-Archimedean field.
In fact, the construction itself will
take place in Berkovich space, and Berkovich dynamics will allow us to tailor the construction so that  $[f^n_t]$ and $[g^n_t]$ converge as $t \to 0$  to distinct elements $[F]$ and $[G]$ of $\ol{\rat}_{d^n}$. To certify that these conjugacy classes are distinct, in most of the cases, we show that the holes of $F$ and $G$ give rise to  non-equivalent markings of $\P^1$.

This paper is mostly devoted to introduce techniques suitable to exploit the interplay between GIT (semi)stability of complex rational maps and dynamics on the Berkovich projective line.
Thus the dynamical content of our constructions is better understood in Berkovich space.
An important relation between Berkovich dynamics of  families as $f^n_t$  and convergence of  $[f^n_t]$  in $\ol{\rat}_{d^n}$ is
addressed in Rumely's work on \emph{(semi)stable reductions} (see \cite[Theorem C]{Rumely17}).
 In Rumely's language, for any family $f_t$ such that $f=f_0$ is semistable but $f_0^n$ is unstable,
the Gauss point is a point of semistable reduction for $f_t$ but not for $f_t^n$. According to Rumely, there exists a type II point in Berkovich space
such that $f_t^n$ has semistable reduction there.
Moreover, if the reduction of $f^n_t$
at a point is stable, then it is the unique point in
Berkovich space with semistable reduction. 
However we were unable to apply directly Rumely's work to gain the required control at points of stable reduction.
Our constructions can be rephrased in this language by saying that given $f\in \cU_n$ we construct several
families $g_t$ with reduction $f$ at the Gauss point
such that $g^n_t$ has stable reduction at a point $\z_0\in\mathbf{P}^1$. The construction
is such that we know where in Berkovich space is the point $\z_0$ and
have some control on the reduction of  $g^n_t$ at $\z_0$.
{More precisely, we have the following result. For the definition of induced maps, see Section \ref{subsection:stable}.}
\begin{introtheorem}\label{Thm:perturbation}
Let $d\ge 4$. Suppose $f\in\cU_n$ and the induced map $\hat f$ is nonconstant.
Then there exists $g_{\lambda,t} \in \C[\lambda, t] (z)$ such that for
$\lambda$ in the complement of a finite subset of $\C$ the following holds:
\begin{enumerate}
\item
$g_{\lambda,t}$ has semistable reduction $f$ at the Gauss point $\xi_g$.
\item
There exists a type II point $\z_0 \neq \xi_g$, independent of $\lambda$, in $\poneberk$ such that
$g^n_{\lambda,t}$ has  stable reduction $G_\lambda$ at $\z_0$.
\item
$[G_\lambda]\in \overline{\rat}_{d^n}$ is not a constant function of $\lambda$.
\item
The action of $g_{\lambda,t}$ in the convex hull of
$\{ \z_0, \dots, g^n_{\lambda,t}(\z_0) \}$ is independent of $\lambda$.
\end{enumerate}
\end{introtheorem}
We point out that in the case that $d=3$, if $\deg\hat f=1$, we also obtain the same conclusion as the above theorem. If $\deg\hat f=2$, we construct two families $f_t$ and $g_t$ satisfying the property stated in the previous paragraph, see Sections \ref{cubic-poly} and \ref{cubic-monomial}.


\medskip
This paper is organized as follows:

In Section \ref{background}, we introduce the relevant preliminaries about degenerate rational maps and Berkovich spaces.
Not all of the material here is standard. In particular,  Sections~\ref{reductions} and~\ref{depths-and-multiplicities} establish a bridge between
Berkovich dynamics and degenerate rational maps which is exploited throughout the paper.

In Section \ref{bad-hole-s}, we identify a distinguished hole of $n$-unstable maps which we call the \emph{bad hole} and establish a basic depth-multiplicity inequality for this hole.
The orbit, depth and multiplicity of the bad hole will organize the proof of Theorem~\ref{main} in cases.

A degenerate rational map $f$ of degree $d$ induces a rational map $\hat{f}$ of lower degree acting on $\P^1$.
Section~\ref{deep},  which concentrates most of the work of the paper, contains the proof of Theorem \ref{Thm:perturbation} and that the GIT-class $[f]$ of any $n$-unstable map $f$
with non-constant induced map $\hat{f}$ lies in $I(\Phi_n)$. To evidence the indeterminacy of $\Phi_n$ at $[f]$, we organize our argument into cases according to
the degree of $\hat{f}$, the depth of the bad hole and the dynamics of the bad hole under $\hat{f}$.
In Section~\ref{constant}, we show that $[f] \in I(\Phi_n)$ provided $f$ is a degenerate semistable rational map with constant induced map and $f \in I(d)\cup \cU_n$. These includes all
the cases not covered in~Section~\ref{deep} and finishes the proof of Theorem~\ref{main}.

\medskip

\textbf{Acknowledgements.}
The work was initiated during the visit of the second author to the Facultad de Matem\'aticas, Pontificia Universidad Cat\'olica de Chile in 2017. He thanks the Facultad de Matem\'aticas for its hospitality.

\section{Preliminaries}\label{background}
In this section we discuss background material and stablish some useful
results about degenerate rational maps, the (GIT) stable and semistable loci of rational maps and Berkovich dynamics.
In Section~\ref{git-ss}, following DeMarco, we focus on  degenerate rational maps $f$, their induced map $\hat{f}: \P^1 \to \P^1$, the holes of $f$ and their depths, as well as the numerical criteria for (semi)stability in terms of holes and depths. In Section~\ref{upper-bound}, we establish Proposition~\ref{semistable}  which states that if $[f] \in I(\Phi_n)$ then $f\in I(d)\cup\cU_n$.  We introduce the basic background on Berkovich dynamics with emphasis on the behavior of the surplus multiplicity  in Section~\ref{berkovich}. After discussing reductions in Section~\ref{reductions},  the fundamental relations between Berkovich dynamics and degenerate complex rational maps are established in Section~\ref{depths-and-multiplicities}. Namely we
relate the  depths and holes of reductions with surplus multiplicities and Berkovich dynamics. In Section~\ref{perturbation}, we state and prove a simple perturbation lemma for rational maps in  Berkovich space which plays a key role in our constructions. Finally, in Section~\ref{complex-action} we briefly discuss the action of complex rational maps on the Berkovich projective line.

\subsection{Stable and semistable rational maps}\label{subsection:stable}
\label{git-ss}
We identify the elements of $\P^{2d+1}$, via coefficients,
with pairs of degree $d$ homogeneous polynomials in two variables modulo
scalar multiplication. That is,
we regard $[P:Q]$ as elements $[a_d:\cdots:a_0:b_d:\cdots:b_0]$ of $\P^{2d+1}$, where $P$ and $Q$ are the degree $d$ homogeneous polynomials
\begin{eqnarray*}
 P(X,Y) & = & a_d X^d + a_{d-1} X^{d-1} Y + \cdots + a_0 Y^d,\\
Q(X,Y) & = & b_d X^d + b_{d-1} X^{d-1} Y + \cdots + b_0 Y^d.
\end{eqnarray*}

The space $\Rat_d$ of degree $d$ rational maps  corresponds to
all $f=[P:Q] \in  \P^{2d+1}$ such that $P$ and $Q$ are relatively prime.
Equivalently, the resultant of $P$ and $Q$, denoted by $\mathrm{Res}(P,Q)$,
does not vanish. Via the identification of $[X:Y] \in \P^1$ with
$z = X/Y \in \C \cup\{\infty\}$ we work, according to convenience, in homogenous or non-homogeneous coordinates.

For $f = [P:Q] \in \P^{2d+1}$, following DeMarco~\cite{DeMarco05},  we will consistently write
\begin{equation}
 \label{hat map}
f= H_f \cdot \hat{f} = H_f \cdot [\widehat{P}: \widehat{Q}]
\end{equation}
where $H_f = \gcd(P,Q)$ and $P = H_f \widehat{P}$, $Q = H_f \widehat{Q}$.
Note that the rational map $\hat{f} = [\widehat{P}: \widehat{Q}]$, called the \textit{induced map of $f$}, may have any degree
$\hat{d}$ with $0 \le \hat{d} \le d$.
It has degree $d$ exactly when $f \in \Rat_d$.
If  $\mathrm{Res}(P,Q)=0$ or equivalently $\hat{d} \le d-1$, then we say that
$f$ is a \emph{degenerate rational map}. In this case,
 the zeros of $H_f$ are called the \emph{holes of $f$}. The set of holes of $f$ is denoted by $\mathrm{Hole}(f)$.
The multiplicity $d_z(f)$ of $z \in \P^1$ as a zero of $H_f$ is called the \emph{depth of $z$}. So $z\in\mathrm{Hole}(f)$ if and only if $d_z(f)\ge 1$.

The action of $\mathrm{PSL}_2(\mathbb{C})$ by conjugation on $\mathrm{Rat}_d$ extends to $\mathbb{P}^{2d+1}$. Geometric invariant theory provides us
with   the \emph{stable and semistable loci}, denoted by $\mathrm{Rat}_d^s$ and $\mathrm{Rat}_d^{ss}$, respectively. Both, the stable and
the semistable locus are $\mathrm{PSL}_2(\mathbb{C})$-invariant. Moreover,
 $\mathrm{Rat}_d \subset \mathrm{Rat}_d^s \subset \mathrm{Rat}_d^{ss}
\subset \mathbb{P}^{2d+1}$.  The quotient of $\mathrm{Rat}_d^s$ by the $\mathrm{PSL}_2(\mathbb{C})$-action is a quasiprojective variety where $\rat_d$ embeds naturally. However, in order to obtain a (compact) projective variety containing  $\rat_d$ the semistable locus is taken into account.
That is, we say that two semistable rational maps  $f,g\in\mathrm{Rat}_d^{ss}$ are GIT conjugate if the Zariski closures of their $\mathrm{PSL}_2(\mathbb{C})$-orbits have common points. For $f\in\mathrm{Rat}_d^s$ the GIT conjugacy coincides with $\mathrm{PSL}_2(\mathbb{C})$-conjugacy. The categorical quotient $\overline{\mathrm{rat}}_d:=\mathrm{Rat}_d^{ss}//\mathrm{PSL}_2(\mathbb{C})$, which set theoretically is formed by GIT conjugacy classes,  is a projective variety that gives us a  natural compactification $\ol{\mathrm{rat}}_d$ of the moduli space $\mathrm{rat}_d:=\mathrm{Rat}_d/\mathrm{PSL}_2(\mathbb{C})$. We simply say that  \emph{ $\ol{\mathrm{rat}}_d$ is the GIT compactification of $\mathrm{rat}_d$}.

The following equivalent stability criteria are due to Silverman and DeMarco, respectively.
\begin{proposition}[{\cite[Proposition 2.2]{Silverman98}}]
Let $f\in\mathbb{P}^{2d+1}$. Then
\begin{enumerate}
\item $f\not\in\mathrm{Rat}_d^{ss}$ if and only if there exists $M\in\mathrm{PSL}_2(\mathbb{C})$ such that  $a_i=0$ for all  $i\ge (d+1)/2$ and $b_i=0$ for all $i\ge (d-1)/2$ where  $M^{-1}\circ f\circ M=[a_d:\cdots:a_0:b_d:\cdots:b_0]$.
\item $f\not\in\mathrm{Rat}_d^s$ if and only if there exists $M\in\mathrm{PSL}_2(\mathbb{C})$ such that  $a_i=0$ for all  $i>(d+1)/2$ and for all $b_i=0$ for $i>(d-1)/2$ if we write $M^{-1}\circ f\circ M=[a_d:\cdots:a_0:b_d:\cdots:b_0]$.
\end{enumerate}
\end{proposition}

\begin{proposition}[{\cite[Section 3]{DeMarco07}}]\label{stability-depth}
Let $f=H_f\hat f\in\mathbb{P}^{2d+1}$. Then
\begin{enumerate}
\item $f \in\mathrm{Rat}_d^{ss}$ if and only if the depth $d_z(f)\le (d+1)/2$ for all $z\in\mathbb{P}^1$, and if $d_h(f)\ge d/2$ for some $h\in\mathbb{P}^1$, then $\hat f(h)\neq h$.
\item $f \in\mathrm{Rat}_d^s$ if and only if the depth $d_z(f)\le d/2$ for all $z\in\mathbb{P}^1$, and if $d_h(f)\ge (d-1)/2$ for some $h\in\mathbb{P}^1$, then $\hat f(h) \neq h$.
\end{enumerate}
\end{proposition}

It follows that the behavior of the depths of the holes under
iteration is relevant to the study the indeterminacy locus of $\Phi_n$.
According to DeMarco~\cite{DeMarco05}, for all $n \ge 2$, the indeterminacy locus $I(d)$
of the iteration map $\Psi_n : \mathbb{P}^{2d+1}\dasharrow\mathbb{P}^{2d^n+1}$ defined by $\Psi_n (f) = f^n$ is independent of $n$ and characterized as
$$I(d)=\{ f\in\P^{2d+1}: \hat f \equiv c\in\P^1, c \in \operatorname{Hole}(f)\}.$$

A formula for the iterates of a map outside $I(d)$ as
well as for the depths of its holes is the content of the next lemma. In the sequel, given a complex rational map $g$, we denote by $m_z(g)$ the
\emph{multiplicity of $g$ at $z\in\mathbb{P}^1$}. {That is, $m_z(g)$ is the number of
  preimages in a neighborhood of $z$ of a generic
  point $w$ close to $g(z)$.}

\begin{lemma}[{\cite[Lemma 2.2]{DeMarco05} and \cite[Lemma 2.4]{DeMarco07}}]
\label{depth-iteration}
If $f\in\mathbb{P}^{2d+1}\setminus I(d)$, then
$$f^n=(\prod_{k=0}^{n-1}(H_f\circ\hat f^k)^{d^{n-k-1}})\hat f^n.$$
Moreover, for all $z\in\mathbb{P}^1$,
\begin{equation}
\label{depth-iterate}
d_z(f^n)=d^{n-1}d_z(f)+\sum_{k=1}^{n-1}d^{n-1-k} m_z(\hat f^k) \, d_{\hat f^k(z)}(f).
\end{equation}
\end{lemma}

The above lemma and the stability criteria suggest that it is useful to work with  the \textit{proportional depths}
$$\ol{d}_z(f) = \dfrac{d_z(f)}{\deg f}$$
and the \textit{proportional multiplicities}
$$\ol{m}_z(\hat f) = \dfrac{m_z(\hat f)}{\deg f}.$$
It follows that
$$\ol{d}_z(f^n)=\ol{d}_z(f)+\sum_{k=1}^{n-1} \ol{m}_z(\hat f^k)\, \ol{d}_{\hat f^k(z)}(f).$$
After remarking that for all $z$,
$$\ol{m}_z(\hat{f}^0 = \mathrm{id}) =1,$$
the above Formula (\ref{depth-iterate}) simply becomes
\begin{equation}
\label{proportional-depth-iterate}
\ol{d}_z(f^n)=\sum_{k=0}^{n-1} \ol{m}_z(\hat f^k) \, \ol{d}_{\hat f^k(z)}(f).
\end{equation}\par

It is also convenient to introduce a notation for the proportional depths thresholds for stability and semistability. That is, for $d\ge 2$, define
$$\mu^-(d):=
\begin{cases}
\dfrac{1}{2} & \text{if}\ d\ \text{is even},\\
\\
\dfrac{d-1}{2d} & \text{if}\ d\ \text{is odd},
\end{cases}$$
and
$$\mu^+(d):=
\begin{cases}
\dfrac{1}{2} & \text{if}\ d\ \text{is even},\\
\\
\dfrac{d+1}{2d} & \text{if}\ d\ \text{is odd}.
\end{cases}$$
Then we may write the stability criteria in terms of
the proportional depths as follows:

\begin{proposition}\label{stability-proportional-depth}
Let $d\ge 2$, $n\ge 1$ and $f \in \mathbb{P}^{2d+1}$ with induced map $\hat{f}$. Then
\begin{enumerate}
\item $f \in\mathrm{Rat}_{d}^{ss}$ if and only if the proportional depth $\ol d_z(f)\le\mu^+(d)$ for all $z\in\mathbb{P}^1$, and if $\ol d_h(f)= \mu^+(d)$
for some $h\in\mathbb{P}^1$, then $\hat f(h)\neq h$.
\item $f \in\mathrm{Rat}_{d}^{s}$ if and only if the proportional depth $\ol d_z(f)\le\mu^-(d)$ for all $z\in\mathbb{P}^1$, and if $\ol d_h(f)=\mu^-(d)$ for some $h\in\mathbb{P}^1$, then $\hat f(h)\neq h$
\end{enumerate}
\end{proposition}
\begin{proof}
It immediately follows Proposition \ref{stability-depth} since $\deg f=d$.
\end{proof}

\subsection{Upper bound for $I(\Phi_n)$}
\label{upper-bound}
As mentioned in the introduction, $[f] \in I(\Phi_n)$ if and only if
$\Phi_n$ has no continuous extension to $[f]$. In fact, by
definition of the indeterminacy locus, if $[f] \in I(\Phi_n)$, then there is
no regular map $\widetilde{\Phi}_n$ defined on a neighborhood $U$ of $[f]$ which agrees with
$\Phi_n$ in the open set where $\Phi_n$ is naturally defined.
Then obviously the lack of a continuous extension
of $\Phi_n$ at $[f]$
implies $[f]\in I(\Phi_n)$. Conversely, noting that $\overline{\rat}_d$ is
a normal variety (see \cite[Theorem 2.1]{Silverman98}) and applying the Zariski's main theorem \cite[Section  III.9]{Mumford88}, we have that
$[f]\in I(\Phi_n)$ implies that $\Phi_n$ has no continuous extension
at $[f]$. Indeed, by contradiction,
suppose $\Phi_n$ extends  continuously at $[f]$. Using the graph of $\Phi_n$ and the
projection onto the first coordinate, we conclude that the graph has
an isolated point above $[f]$. By Zariski's Main Theorem, it follows
that there is a local isomorphism between a neighborhood of $[f]$ and
the graph. The projection $\widetilde{\Phi}_n$ of the local isomorphism onto the
second coordinate coincides with $\Phi_n$, which implies $[f]\not\in
I(\Phi_n)$, that is a contradiction.

The following result extends Lemma 4.2~in~\cite{DeMarco07} and implies that if $[f]\in I(\Phi_n)$ then $f\in I(d)\cup\cU_n$ in Theorem~\ref{main}.

\begin{proposition}\label{semistable}
Suppose $f\in\mathbb{P}^{2d+1}\setminus I(d)$. If $f^n\in\mathrm{Rat}_{d^n}^{ss}$ for some $n>1$, then $f\in\mathrm{Rat}_d^{ss}$ and the iterate map $\Phi_n$ is continuous at $[f]$.
\end{proposition}
\begin{proof}
From \cite[Lemma 4.2]{DeMarco07}, we may assume $d$ is odd and $f^n\in\mathrm{Rat}_{d^n}^{ss}\setminus\mathrm{Rat}_{d^n}^s$. By contradiction, suppose that $f\not\in\mathrm{Rat}_d^{ss}$. According to Proposition \ref{stability-depth}, there would exist $z\in\mathbb{P}^1$ such that $d_z(f)\ge (d+1)/2$. By Lemma \ref{depth-iteration}, we would have $d_{z}(f^n)\ge (d^{n}+d^{n-1})/2$ which is a contradiction with  $f^n\in\mathrm{Rat}_{d^n}^{ss}$.\par

The continuity of $\Phi_n$ at $[f]$ is a direct consequence of the continuity of $\Psi_n : f \mapsto f^n$ at  $f\not\in I(d)$ together with
the fact that the semistable loci are open.
\end{proof}

As an immediate consequence we have:
\begin{corollary}
  \label{upper-bound-c}
  If $[f] \in I(\Phi_n)$, then $f \in I(d)\cup\cU_n$.
\end{corollary}

\subsection{Berkovich spaces}
\label{berkovich}
In this section we briefly summarize some notions and notations regarding
the Berkovich projective line. For more details, we refer the reader to \cite{Baker10, Benedetto19, Berkovich90, Faber13I, Jonsson15, Kiwi15}.

The algebraic closure of   the field of  formal Laurent series
$\C ((t))$ with coefficients in $\mathbb{C}$ is
the field $\ol{\C ((t))}$ of formal Puiseux series.
It is naturally endowed with a valuation $\mathrm{ord}(\cdot)$ given by the order
of vanishing at $t=0$ and with its associated
non-Archimedean absolute value  $|z| = \mathrm{e}^{-{\mathrm{ord}(z)}}$.
Let $\mathbb{L}$ be the completion of the field of  Puiseux series.
Write $\cO_\L$ for the ring of integers and $\mathfrak{M}_\L$ for the maximal ideal.
Then the residue field $\cO_\L/ \mathfrak{M}_\L$ is canonically identified with $\C$.

For $ r \ge 0$ and $z \in \L$, let  $B_r (z) = \{ w \in \L : |w -z| \le r \}$ and
$B_r^- (z) = \{ w \in \L : |w -z| < r \}$. When $r \notin \mathrm{e}^\Q = |\L^\times|$
these balls coincide: $B_r (z) = B_r^- (z)$. Although both are clopen sets in the metric topology, we say that $B_r (z)$ is a \emph{closed disk} and $B_r^- (z)$ is an \emph{open disk}.

The Berkovich projective line  $\mathbf{P}^1$ is a connected compact Hausdorff topological space which contains $\P^1_\L$ as a dense subset \cite[Proposition 2.6 and Lemma 2.9]{Baker10}.
It consists of $4$ types of points. After identification of
$\mathbb{L} \cup \{\infty\}$ with $\P^1_\L$ these types can be described as follows. The points of the projective space $\mathbb{P}^1_{\mathbb{L}}$ are the type I points. The type II  (resp. type III) points correspond to closed disks in $\mathbb{L}$ with radii in (resp. not in) the value group $|\mathbb{L}^\times|$.
 The type IV points are related to a decreasing sequence of closed disks in $\mathbb{L}$ with empty intersection.

Given a disk  $B = B_r (z)$ we will denote the associated point either by
$\xi_B$ or $\xi_{z,r}$ according to convenience.
The type II point  $\xi_{0,1}$  associated to the closed unit disk containing  $z=0$
is called the \emph{Gauss point} and simply denoted by $\xi_g$.

The space $\mathbb{H}_\mathbb{L}:=\mathbf{P}^1\setminus\mathbb{P}^1_{\mathbb{L}}$ admits a natural hyperbolic metric, see \cite[Section 2.7]{Baker10}. We denote by $\rho(\xi_1, \xi_2)$ the hyperbolic distance of two points $\xi_1,\xi_2\in\mathbb{H}_\mathbb{L}$. With this metric $\mathbb{H}_\mathbb{L}$
is a metric $\R$-tree with endpoints at infinity parametrized by  $\P^1_\L$.  However, the metric topology of  $\mathbb{H}_\mathbb{L}$ is stronger than the subspace topology of $\poneberk$.
In fact, $\mathbb{H}_\mathbb{L}$ is not locally compact in the metric topology.

 For $\xi\in\mathbf{P}^1$, the tangent space $T_{\xi}\mathbf{P}^1$ is the set of  connected components of $\mathbf{P}^1\setminus\{\xi\}$. Each element $\vec{v}\in T_{\xi}\mathbf{P}^1$ is called a \emph{tangent vector at $\xi$} and the corresponding connected component is denoted by $\mathbf{B}_{\xi}^-(\vec{v})$. At each type II point $\xi$, the tangent space $T_\xi\mathbf{P}^1$ can be identified with the complex projective line $\mathbb{P}^1$ \cite[Section 3.8.7]{Jonsson15}.  At the Gauss point $\xi_g$, this identification is canonical.  Namely,
each direction at $\xi_g$ contains a unique point $z \in \P^1 \subset \P^1_\L$.

Now consider a rational map $\phi \in \L(z)$. Then  $\phi: \P_\L \to \P_\L$ has a unique continuous extension to Berkovich space
$\phi:\mathbf{P}^1\to\mathbf{P}^1$ \cite[Section 2.3]{Baker10}. At each point $\xi\in\mathbf{P}^1$, the map $\phi$ has a well defined local degree $\deg_\xi\phi$ \cite[Proposition 9.28]{Baker10}.
Moreover, if $\xi$ is a type II point, it induces a \emph{tangent map}
$T_{\xi}\phi:T_{\xi}\mathbf{P}^1\to T_{\phi(\xi)}\mathbf{P}^1$ which is a rational map of degree $\deg_\xi \phi$ in the corresponding $\P^1$-structures, see \cite[Theorem 9.26]{Baker10}.

For each point $\xi\in\mathbf{P}^1$ and each tangent vector $\vec{v}\in T_{\xi}\mathbf{P}^1$, there exist two well defined multiplicities:
the \emph{directional multiplicity $m_\phi ( \vec{v}) \ge 1$} and the
\emph{surplus multiplicity $s_\phi(\vec{v}) \ge 0$} characterized as follows. A point in  $\mathbf{B}_{\phi(\xi)}^-(T_{\xi}\phi(\vec{v}))$ has exactly $m_\phi( \vec{v})+s_\phi(\vec{v})$  preimages, counting multiplicities, in  $\mathbf{B}_{\xi}^-(\vec{v})$ and a point in the complement of  $\mathbf{B}_{\phi(\xi)}^-(T_{\xi}\phi(\vec{v}))$  has exactly $s_\phi( \vec{v})$ preimages, counting multiplicities, in $\mathbf{B}_{\xi}^-(\vec{v})$, see \cite[Proposition 9.41]{Baker10},\cite[Proposition 3.10]{Faber13I} and \cite[Lemma 2.1]{Rivera03II}.
If $\xi$ is a type II point, then $m_\phi(\vec{v})$ coincides with the multiplicity of $T_\xi \phi$ at $\vec{v}$.
Moreover, for all $\xi \in \poneberk$,
\begin{equation}
  \label{surplus-sum}
  d = \deg_\xi \phi + \sum_{\vec{v} \in T_{\xi}\mathbf{P}^1} s_\phi(\vec{v}).
\end{equation}

\begin{lemma}\label{surplus-composition}
Let $\phi, \psi \in\mathbb{L}(z)$ be non-constant rational maps. Then for any $\xi\in\mathbf{P}^1$ and $\vec{v}\in T_{\xi}\mathbf{P}^1$,
$$s_{\psi \circ \phi} (\vec{v}) = \deg \psi \cdot s_\phi (\vec{v}) + s_\psi (T_\xi\phi (\vec{v})) \cdot m_\phi (\vec{v}).$$
\end{lemma}

\begin{proof}
  Let $\xi_0 = \xi, \xi_1 = \phi(\xi)$ and $\xi_2 = \psi \circ \phi (\xi)$.
  Similarly, let $\vec{v}_0 = \vec{v}, \vec{v}_1 = T_\xi\phi(\vec{v})$ and $
\vec{v}_2 = T_{\xi_1}\psi(\vec{v}_1)$.
  Given $x \notin \mathbf{B}_{\xi_2}^-(\vec{v}_2)$, out of the $\deg \psi$ preimages under $\psi$
of $x$ there are exactly $s_\psi(\vec{v}_2)$ in $\mathbf{B}_{\xi_1}^-(\vec{v}_1)$.
Each of these $s_\psi(\vec{v}_2)$ points has
 $m_\phi( \vec{v})+s_\phi(\vec{v})$ preimages under $\phi$ in
$\mathbf{B}_{\xi}^-(\vec{v})$.
Each of the $\deg \psi - s_\psi(\vec{v}_2)$ preimages under $\psi$
of $x$ which are not in $\mathbf{B}_{\xi_1}^-(\vec{v}_1)$ has
 exactly $s_\phi(\vec{v})$ preimages in  $\mathbf{B}_{\xi}^-(\vec{v})$.
Thus the total number of preimages of $x$ in $\mathbf{B}_{\xi}^-(\vec{v})$ is
$$s_\psi(\vec{v}_2) \cdot (m_\phi( \vec{v})+s_\phi(\vec{v})) +
(\deg \psi - s_\psi(\vec{v}_2)) \cdot s_\phi(\vec{v}).$$
\end{proof}

Observe that the previous lemma suggests that it is also nicer in this context to work with the \emph{proportional multiplicities} defined as follows:
$$\ol{s}_\phi(\vec{v}) := \dfrac{s_\phi (\vec{v})}{\deg \phi},$$
and
$$\ol{m}_\phi(\vec{v}) := \dfrac{m_\phi (\vec{v})}{\deg \phi}.$$
With this notation the formula of the lemma becomes:
$$\ol s_{\psi \circ \phi} (\vec{v}) = \ol{s}_\phi (\vec{v}) + \ol{s}_\psi (T_\xi\phi (\vec{v})) \cdot \ol{m}_\phi (\vec{v}).$$

Now we consider the behavior of surplus multiplicities under iteration.
When the map $\phi$ is clear from context we lighten notation and simply write $s(\vec{v})$ for $s_\phi(\vec{v})$ and $m( \vec{v})$ for $m_\phi( \vec{v})$. Moreover, for $k \ge 1$,
we write
\begin{eqnarray*}
  s^k (\vec{v}) & := & s_{\phi^k} (\vec{v}), \\
  \ol{s}^k (\vec{v}) & := & \ol{s}_{\phi^k} (\vec{v}),\\
m^k (\vec{v}) & := & m_{\phi^k} (\vec{v}),\\
\ol{m}^k (\vec{v}) & := & \ol{m}_{\phi^k} (\vec{v}).
\end{eqnarray*}
For $k=0$ we agree that $m^0 = \ol{m}^0 = 1$.

\begin{lemma}\label{surplus-iterate}
Let $\phi \in\mathbb{L}(z)$ be a rational map of degree $d \ge 1$.Then for any $\xi\in\mathbf{P}^1$ and $\vec{v}\in T_{\xi}\mathbf{P}^1$,
$$s^n ( \vec{v})=d^{n-1}s(\vec{v})+\sum_{k=1}^{n-1}m^k(\vec{v}) \cdot s( T_\xi\phi^k(\vec{v}))d^{n-1-k},$$
Equivalently,
\begin{equation}
\label{surplus-iteration}
\ol{s}^n ( \vec{v})= \sum_{k=0}^{n-1} \ol{m}^k(\vec{v}) \cdot \ol{s}( T_\xi\phi^k(\vec{v})).
\end{equation}
\end{lemma}
\begin{proof}
Apply induction after observing that from the previous lemma we have
$$s^n (\vec{v}) = d s^{n-1} (\vec{v}) + m^{n-1} (\vec{v}) s( T_\xi\phi^{n-1}(\vec{v})).$$
\end{proof}

\subsection{Reductions}
\label{reductions}
Under the canonical identification of the residue field  $\cO_\L/\mathfrak{M}_\L$ with $\C$, given
$a \in \cO_\L$ we denote by $\tilde{a} \in \C$ its reduction $\mod \mathfrak{M}_\L$.

A rational map $\phi$ in $\mathbb{L}(z)$ of degree $d$ is naturally identified with
an element of $\P^{2d+1}_\L$ via its coefficients. In homogenous coordinates we may write
$\phi([X:Y]) = [F(X,Y): G(X,Y)]$ where
\begin{eqnarray*}
  F(X,Y)& =& \sum a_j X^j Y^{d-j},\\
  G(X,Y) &= &\sum b_j X^j Y^{d-j}
\end{eqnarray*}
for some $a_j,b_j \in \L$.
We identify $\phi$ with
$[a_d : \dots :a_0 : b_d: \dots :b_0] \in \P^{2d+1}_\L$ and also write
$\phi=[F:G] \in  \P^{2d+1}_\L$.

There are two related notions of ``reductions''
of $\phi$, one as a \emph{map} which we will denote by $\tilde{\phi}$,
and the other by its \emph{coefficients} which will be denoted by $\phi_0$.
To introduce both reductions we first consider a \emph{normalized representation of $\phi$}
which, in homogenous coordinates, consists on scaling $\phi$ by a suitable element
of $\L$ so that we have $\phi = [F: G]$ where  $F, G \in \cO_\L [X,Y]$ with at least one coefficient being a unit. Equivalently, it consists on scaling in order to write
 $\phi= [a_d : \dots :a_0 : b_d: \dots :b_0]$ where $a_j,b_j \in \cO_\L$ with at least one entry of absolute value $1$.

We say that $$\phi_0 =[\widetilde{F}: \widetilde{G}] = [\widetilde{a_d} : \dots :\widetilde{a_0} : \widetilde{b_d}: \dots :\widetilde{b_0}]
\in \P^{2d+1}$$ is the \emph{coefficient reduction} of $\phi$. The coefficient reduction is independent of the  normalized representation of $\phi$.
Note that the coefficient reduction is just the one induced by reduction on
parameter space, that is, the natural reduction from  $\P^{2d+1}_\L$ onto $\P^{2d+1}$.

Following Rumely, our notion of coefficient
reduction is a particular case
of a more general notion of reduction at type II points.
Indeed, given $\psi \in \L(z)$, a type II point $\z_0 \in \mathbf{P}^1$ and a M\"obius transformation $M \in \L(z)$ such that $M(\xi_g) = \z_0$, we say that the coefficient reduction $f$ of $M \circ \psi \circ M^{-1}$ is a \emph{reduction of $\psi$ at $\z_0$}. This reduction is unique up to conjugacy by a M\"obius transformation in $\C(z)$.  The coefficient reduction introduced above corresponds to reduction at the Gauss point. It follows that a reduction of $\psi$ at
a point $\z_0$ is stable, semistable or unstable independently of the choice of $M$.

Let $H_{\phi_0} (X,Y) = \gcd (\widetilde{F}(X,Y) , \widetilde{G} (X,Y) )$  and consider
$\widehat{F} ,\widehat{G} \in \C[X,Y]$ such that  $\widetilde{F}
= H_{\phi_0} \cdot \widehat{F}$ and $\widetilde{G} = H_{\phi_0} \cdot \widehat{G}$. Then we say that $\tilde{\phi} =
[\widehat{F} :\widehat{G}]$ is the \emph{reduction of $\phi$}.
Note that $\tilde{\phi}$ is induced by reduction on dynamical space, that is, by the natural projection $\P^{1}_\L \to \P^{1}$.

With the notation of Section~\ref{git-ss}, we have that
$$\phi_0 = H_{\phi_0} \cdot \tilde{\phi}.$$
Thus, the induced map of the coefficient reduction is the reduction map:
 $$\widehat{\phi}_0 = \tilde{\phi}.$$

\subsection{Depths and multiplicities}
\label{depths-and-multiplicities}
Depths of holes and surplus multiplicities are closely related when we consider holomorphic families of rational maps as dynamical systems acting on the Berkovich projective line.
Given a neighborhood $V$ of $t=0$ in $\C$, we say that a family $\{f_t\}\subset\mathbb{P}^{2d+1}$ parametrized by $t \in V$  is a \emph{holomorphic family of rational maps} if the map $V \to\mathbb{P}^{2d+1}$, sending $t$ to $f_t$, is holomorphic and $f_t\in\mathrm{Rat}_d$ for all $t\not=0$. If $f_0 \notin \mathrm{Rat}_d$, we say that the family $\{f_t\}$ is a \emph{degenerate} holomorphic family of rational maps.

A holomorphic family $\{f_t\}$ of degree $d\ge 1$ complex rational maps induces a rational map $\mathbf{f}:\mathbf{P}^1\to\mathbf{P}^1$
since $f_t = [F_t:G_t]$ where $F_t$ and $G_t$ are homogenous polynomials in two variables with coefficients given by holomorphic functions in $t$. In particular, the coefficients
are Taylor series in $t$ and thus we may regard the family $\{f_t\}$ as a rational map
$\mathbf{f} \in \L(z)$.
The coefficient reduction of $\mathbf{f}$
is precisely $f_0$. We will systematically abuse of notation also writing $f_t$ for
the rational map $\mathbf{f}$ with coefficients in $\L$ and
  $f_t : \poneberk \to \poneberk$ for its
action on Berkovich space.

\begin{lemma}[{\cite[Lemma 2.17]{Baker10} and \cite[Lemma 3.17]{Faber13I}}] \label{depth-surplus}
Let $\phi \in\mathbb{L}(z)$ be a  rational map.
Then the reduction map $\tilde{\phi}$ is non-constant if and only if $\phi(\xi_g) = \xi_g$.
In this case, under the canonical identification $z \leftrightarrow \vec{v}_z$ of
$\P^1$ with $T_{\xi_g} \mathbf{P}^1$ we have that
$T_{\xi_g} \phi = \tilde{\phi}$. Moreover,
\begin{eqnarray*}
  d_z (\phi_0) &= & s_\phi (\vec{v}_z),\\
  m_z (\tilde{\phi}) &= & m_\phi (\vec{v}_z).
\end{eqnarray*}
\end{lemma}

If the reduction map is constant, we have

\begin{lemma}\label{main-lemma}
Consider a degree $d$ rational map $\phi \in\mathbb{L}(z)$ such that
$\phi(\xi_g) = \xi \neq \xi_g$.
Let $\vec{w}$ be the direction at $T_\xi \poneberk$ containing the Gauss point.
Then
$$
d_z(\phi_0)=
\begin{cases}
  s_\phi(\vec{v}_z) & \mbox{if }\,\,  T_{\xi_g}\phi(\vec{v}_z)\not=\vec{w},\\
  s_\phi(\vec{v}_z) + m_\phi (\vec{v}_z) &  \mbox{if }\,\,  T_{\xi_g}\phi(\vec{v}_z) =\vec{w}.
\end{cases}
$$
\end{lemma}
\begin{proof}
Consider a degree $1$ map $M \in \L(z)$ such that $\xi = M(\xi_g)$.
We claim that $u \in \P^1$ is a hole of $M_0$ if and only if the corresponding
direction $\vec{v}_u$ is such that $\xi_g \in M(\vec{v}_u)$.
Namely we claim the assertion of the lemma for degree $1$ maps.
We assume that the hole $h$ of $M$ is not $\infty$ and proceed
using non-homogenous coordinates. For $h = \infty$ the claim follows along similar lines.
Since $h \neq \infty$,  there exists $a, b  \in \cO_\L$ and $c \in \L$
such that
$$M (z) = c \cdot \dfrac{z-a}{z-b}$$
with
$$M_0 (z) = H_\mu(z) \cdot \alpha,$$
where $H_\mu(z) = z -h$ and $\alpha = \tilde{c}$, $h = \tilde{a} = \tilde{b}$.
It follows that $M$ maps every direction $\vec{v}_z$ with $z \neq h$,
into the direction $\vec{v}_\alpha$ and  therefore $\xi_g \in M(\vec{v}_h)$. The claim easily
follows.

Let $\phi$ be as in the statement of the Lemma and consider
 $\psi (z) = M^{-1} \circ \phi (z) \in \L(z)$
 with non-constant reduction map $\tilde{\psi}$.
It follows that $$\phi_0 (z) =  (M \circ \psi)_0 (z) = H_{\psi_0} (z) \cdot H_\mu ( \tilde{\psi} (z))  \cdot \alpha,$$
where  $\psi_0 = H_{\psi_0} \cdot \tilde{\psi}$.
Therefore, if $H_\mu ( \tilde{\psi} (z)) \neq 0$, we have $d_z(\phi_0)=d_z(\psi_0) = s_\psi(\vec{v}_z)$, otherwise, if $H_\mu ( \tilde{\psi} (z)) = 0$, we have
 $d_z(\phi_0) = d_z(\psi_0) + m_z (\tilde{\psi})$.
Since $s_\phi (\vec{v}) = s_\psi (\vec{v})$ for all $\vec{v}$ and
$m_z (\tilde{\psi}) = m_\psi (\vec{v}_z) = m_\phi(\vec{v}_z)$, it only remains to observe that the above claim says that
 $H_\mu ( \tilde{\psi} (z)) = 0$ if and only if $T_{\xi_g}\phi(\vec{v}_z) =\vec{w}$.
\end{proof}

We apply the previous lemma to study the depths of the holes of the reduction at a general type II point $\z_0$:


\begin{corollary}\label{relative-position}
Consider a rational map $\phi \in \L(z)$ and a type II point $\xi_0 \in \mathbf{P}^1$. Let $\xi_n = \phi^n (\xi_0)$
and let $L$ in $\L(z)$ be an affine map such that $\xi_g = L(\xi_0)$. Set
$$f=(L \circ \phi^n \circ L^{-1})_0.$$
Given $z \in\P^1$, let $\vec{v} \in T_{\xi_0} \poneberk $ be such that  $\vec{v}_z= T_{\xi_0}L(\vec{v})$. Then

$$d_z(f)=
\begin{cases}
  s^n_\phi (\vec{v}) & \mbox{if}\,\,  \xi_0\ \mbox{is not in}\,\, T_{\xi_0} \phi^n (\vec{v}), \\
  s^n_\phi  (\vec{v}) + m^n_\phi  (\vec{v}) & \mbox{if}\,\,  \xi_0\ \mbox{is in}\,\, T_{\xi_0} \phi^n (\vec{v}).
\end{cases}$$
\end{corollary}
\begin{proof}
Set $\varphi=L \circ \phi^n \circ L^{-1}$ and $\xi=\varphi(\xi_g)$. If $\xi=\xi_g$, then the conclusion follows form Lemma \ref{depth-surplus}. If $\xi\not=\xi_g$, let $\vec{w}$ be the direction at $\xi$ containing $\xi_g$. By Lemma ~\ref{main-lemma}, for any direction $\vec{v}_z$ at $\xi_g$, we have
 $$
 d_z(\varphi_0)=
\begin{cases}
  s_\varphi(\vec{v}_z) & \mbox{if }\,\,  T_{\xi_g}\varphi(\vec{v}_z)\not=\vec{w},\\
  s_\varphi(\vec{v}_z) + m_\varphi (\vec{v}_z) &  \mbox{if }\,\,  T_{\xi_g}\varphi(\vec{v}_z) =\vec{w}.
\end{cases}
$$
Since $\xi=L(\xi_n)$, the lemma follows after taking the preimages of the relevant directions under the affine map $L$.
\end{proof}

\subsection{Perturbation of rational maps in Berkovich space}
\label{perturbation}
Our constructions rely on starting with a map $\phi \in \L(z)$ and conveniently increasing its degree by strategically placing new zeros and  poles. We may perform this ``perturbation'' without changing the action of $\phi$ nearby the Gauss point provided that the new zeros and poles are sufficiently close:

\begin{lemma}\label{construction}
Let $\phi\in\mathbb{L}(z)$ be a {nonconstant} rational map. Let $\xi=\xi_{z_0,|t|^\alpha}$ and let $\phi(\xi)=\xi_{w_0,|t|^\beta}$. Suppose $N>\max\{\alpha,\beta,\alpha+\beta\}$ is an integer and $p\in\mathbb{L}$. Consider
$$\psi(z):=\left(1+\frac{t^N}{z-p}\right)\phi(z).$$
Then $\psi(\xi)=\phi(\xi)$ and $T_{\xi}\psi=T_\xi\phi$.
Moreover, provided that $\phi(p) \neq0$, for any $\vec{v} \in T_\xi \poneberk$, if $p$ is in the direction $\vec{v}$
then $s_\psi (\vec{v}) = s_\phi (\vec{v}) +1$, otherwise $s_\psi (\vec{v}) = s_\phi (\vec{v})$.
\end{lemma}

\begin{proof}
Identifying $T_{\xi}\mathbf{P}^1$ with $\P^1$, for all but finitely many $c\in\P^1$, we have
$$\phi(z_0+ct^\alpha+h.o.t.)=w_0+T_{\xi}\phi(c)t^\beta+h.o.t.$$
For such $c$, we have
$$\psi(z_0+ct^\alpha+h.o.t.)=(1+\frac{t^N}{z_0+ct^\alpha+h.o.t.-p})(w_0+T_{\xi}\phi(c)t^\beta+h.o.t.).$$
If $|p-z_0|\le |t|^\alpha$,
\begin{align*}
\psi(z_0+ct^\alpha+h.o.t.) &=(1+\frac{t^N}{c't^\alpha+h.o.t.})(w_0+T_{\xi}\phi(c)t^\beta+h.o.t.)\\
&=w_0+T_{\xi}\phi(c)t^\beta+\frac{w_0}{c'}t^{N-\alpha}+\frac{T_{\xi}\phi(c)}{c'}t^{N-\alpha+\beta}+h.o.t..
\end{align*}
If $|p-z_0|>|t|^\alpha$,
\begin{align*}
\psi(z_0+ct^\alpha+h.o.t.)&=(1+\frac{t^N}{(z_0-p)+ct^\alpha})(w_0+T_{\xi}\phi(c)t^\beta+h.o.t.)\\
&=w_0+T_{\xi}\phi(c)t^\beta+\frac{w_0}{z_0-p}t^{N}+\frac{T_{\xi}\phi(c)}{z_0-p}t^{N+\beta}+h.o.t.
\end{align*}
Since $N>\max\{\alpha,\beta,\alpha+\beta\}$, we have $\psi(\xi)=\phi(\xi)$ and $T_{\xi}\psi=T_\xi\phi$.
Finally, counting the number of preimages of $\infty$ in the direction $\vec{v}$, the lemma follows.
\end{proof}
An immediate corollary of Lemma \ref{construction} is the following.
\begin{corollary}
\label{perturbation-c}
Let $\phi\in\mathbb{L}(z)$ be a non-constant rational map.  Suppose $\Gamma$ is a graph contained in a bounded set in $\mathbb{H}_\mathbb{L}$ with respect to the hyperbolic metric $\rho$. For $p_1,\cdots p_k\in \mathbb{L}$, consider
$$\psi(z):=\phi(z)\prod_{i=1}^k\left(1+\frac{t^N}{z-p_i}\right).$$
Then for sufficiently large $N>0$, we have $\psi(\xi)=\phi(\xi)$ and $T_{\xi}\psi=T_\xi\phi$ for all $\xi\in\Gamma$.
Moreover, suppose that $\phi(p_i) \neq0$ for all $1\le i\le k$. Then
for any $\xi\in\Gamma$ and $\vec{v} \in T_\xi \poneberk$, we have
$$s_\psi (\vec{v}) = s_\phi (\vec{v}) +\#\{i: p_i\in\mathbf{B}_\xi^-(\vec{v})\}.$$
\end{corollary}

\subsection{Action of complex rational maps on Berkovich space}
\label{complex-action}
The starting point of our constructions are complex rational maps $g \in \C(z)$ of degree at least $1$ which we
may regard as  elements of $\L(z)$. The action of such maps on $\poneberk$ is  not difficult to describe. In fact, elementary arguments omitted here show that
for all $z \in \C \subset \P^1_\L$ and $\alpha >0$,  we have that
$$\xi= \xi_{z, |t|^\alpha} \mapsto g(\xi) =\xi_{g(z), |t|^{\alpha m}},$$
where $m= m_z(g)$ is the (complex)
multiplicity of $z$. That is, $g$ is linear in the interval $[\xi_g, z[$ (with respect to the hyperbolic length) with ``slope'' $m_z(g)$. For all $0 < \alpha \in \Q$, the point $\xi$ is of type II, and for all $w \in \C$,     the direction in $T_\xi \poneberk$ containing
$z + w t^\alpha$ is mapped by $T_\xi g$ to the direction containing $g(z) + w^m t^{m \alpha}$  with zero surplus multiplicity. The  direction containing the Gauss point is mapped by
$T_\xi g$ to the direction at $g(\xi)$ containing the Gauss point with surplus multiplicity $\deg g - m$.
The multiplicity of a direction in $T_\xi\poneberk$ is $m$ if it contains $z$ or the Gauss point, and $1$ otherwise. A similar description holds for $z = \infty$.

\section{The bad hole of $n$-unstable maps}
\label{bad-hole-s}

In this section we show that each $n$-unstable map $f$  has a distinguished hole $\mathtt{h}$ where semistability of the  $n$-th iterate fails.
The dynamics and depth of this distinguished hole $\mathtt{h}$, which we call the ``bad hole of $f$'', will organize the proof of Theorem~\ref{main}.
In fact, the next section is devoted to prove that if a $n$-unstable map  has non-constant induced map $\hat f$, then $[f] \in I(\Phi_n)$.
Our proof relies on the construction of holomorphic families through $f$ that confirm the indeterminacy of $\Phi_n$ at $[f]$.
The constructions are organized in five cases according to the dynamics and depth of the bad hole.

\medskip
Recall that the set of $n$-unstable  maps is denoted by $\nUn$.

\subsection{Maps in $\cU_n$ and the bad hole}
For $f\in\cU_n$, the semistability condition in Proposition \ref{stability-depth} for $f^n$ breaks down at a unique hole:

\begin{lemma}\label{bad hole}
If $f\in\cU_n$, then there is a unique $h\in\mathrm{Hole}(f)$ such that $d_h(f^n)\ge d^n/2$.
\end{lemma}

\begin{definition}[Bad Hole]
 Given $f \in \nUn$, we say that the  hole given by Lemma \ref{bad hole} is the \emph{bad hole of $f$ for the $n$-th iterate} and we denote it by $\mathtt{h}$.\par
\end{definition}

\begin{proof}
Consider $f \in \nUn$. If $d$ is odd, then there is a unique hole $h\in\mathrm{Hole}(f^n)$ such that $d_h(f^n)\ge (d^n+1)/2$. Indeed, since the sum of the depth of the holes is at most $d^n$, we have that no other hole has depth at least $d^n/2$.

Now we consider the even degree case. Since $f\not\in I(d)$, then $f^n\not\in I(d^n)$. For otherwise, there exists $h\in\mathrm{Hole}(f^n)$ such that $\widehat{f^n}=h$. By Lemma \ref{depth-iteration}, we have that $h\in\mathrm{Hole}(f)$ and $\hat f=h$. It follows that $f\in I(d)$, which is a contradiction. We claim that $f^n$ has a unique hole $h\in\mathrm{Hole}(f^n)$ with depth $d_h(f^n)\ge d^n/2$. By contradiction, suppose there are two distinct holes in $\mathrm{Hole}(f^n)$ with depths at least $d^n/2$. Then up to conjugacy,
$$f^n([X:Y])=X^{d^n/2}Y^{d^n/2}[1:1].$$
Hence $f^n\in\mathrm{Rat}_{d^n}^s$. So $f\not\in\cU_n$ which is a contradiction. \par
Now we show  $h\in\mathrm{Hole}(f)$. Note
$$\mathrm{Hole}(f^n)=\bigcup_{i=0}^{n-1}\hat f^{-i}(\mathrm{Hole}(f)).$$
In particular, if $\deg\hat f=0$, then $\mathrm{Hole}(f^n)=\mathrm{Hole}(f)$. So we may assume that $\hat{f}$ is not constant.
We proceed by contradiction. Suppose $h\not\in\mathrm{Hole}(f)$. Then
$$\dfrac{d^n}{2} \le d_h ( f^n ) = m_h (\hat{f}) \, d_{\hat{f}(h)} (f^{n-1}) \le d \cdot d_{\hat{f}(h)} (f^{n-1}) \le d_{\hat{f}(h)} (f^{n}).$$
By the already proven uniqueness of the bad hole, $\hat{f}(h) = h$, and hence
$h$ is a hole of $f$ which contradicts $h\not\in\mathrm{Hole}(f)$.
\end{proof}

\begin{corollary}
\label{increasing}
For all $n \ge 2$, we have that $\mathcal{U}_n \subset \mathcal{U}_{n+1}$. Moreover,
if $\mathtt{h}$ is the bad hole for the $n$-th iterate of $f \in \mathcal{U}_n$, then
  $\mathtt{h}$ is the bad hole for the $n+1$-th iterate of $f$.
\end{corollary}

\begin{proof}
  Consider $f \in \mathcal{U}_n$.
  If  $f^n$ has a hole $h$ of depth $d_h (f^n) > (d^n+1)/2 $ then
$d_h (f^{n+1}) > (d^{n+1}+d)/2$. Hence $f^{n+1}\not\in\mathrm{Rat}_{d^{n+1}}^{ss}$.
  If $f^n$ has a hole $h$ with $ (d^n+1)/2 \ge d_h (f^n) \ge d^n/2 $ such that
$\hat f^n (h) = h$, then $h$ is a hole of $f$ and $\hat f$ is non-constant, since $f \notin I(d)$. Moreover, $$d_h (f^{n+1} ) = d_h (f^n) \cdot d +
m_h (\hat f^n) \cdot d_h (f) \ge d^{n+1}/2 + 1.$$
Hence, $f^{n+1}\not\in\mathrm{Rat}_{d^{n+1}}^{ss}$ (i.e. $f \in \mathcal{U}_{n+1}$).
\end{proof}

The following result asserts that the forward orbit of the bad hole meets a hole of $f$ (maybe the bad itself) provided the depth is at most $d/2$.
\begin{proposition}\label{orbit-hole}
Suppose that $f\in\mathcal{U}_n$ has non-constant induced map $\hat f$. If the forward $\hat f$-orbit of the bad hole $\mathtt{h}$ does not intersect $\mathrm{Hole}(f)$, then $d_\mathtt{h}(f)=(d+1)/2$.
\end{proposition}
\begin{proof}
Since $f\in\mathrm{Rat}_d^{ss}$, by Proposition \ref{stability-depth}, we have that $d_\mathtt{h}(f)\le (d+1)/2$. By contradiction, suppose that the forward orbit of $\mathtt{h}$ is hole free and  $d_\mathtt{h}(f)<(d+1)/2$. Then $d_\mathtt{h}(f)\le d/2$. Since $f^n\not\in\mathrm{Rat}^{ss}_{d^n}$  and $\hat f^n(\mathtt{h})\not=\mathtt{h}$, again by Proposition \ref{stability-depth}, we have that $d_\mathtt{h}(f^n)> d^n/2$.  However, by Formula (2), we obtain that $d_{\mathtt{h}}(f^n)=d^{n-1}d_\mathtt{h}(f)\le d^n/2$, which is a contradiction. 
\end{proof}


\subsection{Multiplicity inequality for $n$-unstable maps}

In order to classify $n$-unstable maps, it will be useful to employ
a basic inequality involving the multiplicities of the bad hole.

\begin{lemma}\label{depth-multiplicity-inequality}
  If $f \in\mathcal{U}_n$ and $\mathtt{h}$ is the bad hole of $f$, then
$$2 d_\mathtt{h} (f) + {m_\mathtt{h} (\hat{f})} > d,$$
equivalently
$$2 \ol{d}_\mathtt{h} (f) + {\ol{m}_\mathtt{h} (\hat{f})} > 1.$$
\end{lemma}

\begin{proof}
  Assume that $\hat{f} (\mathtt{h}) = \mathtt{h}$. If $\ol{m}_{\mathtt{h}}(\hat f )\not=1$, then
$$\dfrac{1}2 \le \ol{d}_{\mathtt{h}} (f^n)= \ol{d}_{\mathtt{h}} (f) \cdot \sum_{k=0}^{n-1}
 \ol{m}_{\mathtt{h}}(\hat f^k) = \ol{d}_{\mathtt{h}} (f) \cdot \dfrac{1- \ol{m}_{\mathtt{h}}(\hat f )^n}{1- \ol{m}_{\mathtt{h}}(\hat f )}<\dfrac{\ol{d}_{\mathtt{h}} (f)}{1- \ol{m}_{\mathtt{h}}(\hat f )}.$$
 If $\ol{m}_{\mathtt{h}}(\hat f )=1$, then
$$2 \ol{d}_\mathtt{h} (f) + {\ol{m}_\mathtt{h} (\hat{f})}\ge \frac{2}{d}+1 > 1.$$

Now assume $\hat{f} (\mathtt{h}) \neq \mathtt{h}$. Since
there is a unique bad hole, it follows that $\ol{d}_{\hat f(\mathtt{h})} (f^{n-1}) < 1/2$. Indeed, for otherwise,  
$$\ol{d}_{\hat f(\mathtt{h})} (f^{n}) = \ol{d}_{\hat f(\mathtt{h})} (f^{n-1}) +\ol{m}_{\hat f(\mathtt{h})} (\hat f^{n-1}) \,
\ol{d}_{\hat f^{n-1} (\mathtt{h})}(f) \ge 1/2,$$
which implies that $\hat f(\mathtt{h})$ is the bad hole.
Therefore, {if $\hat f$ is nonconstant, we have}
$$\dfrac{1}2 \le \ol{d}_\mathtt{h} (f^n) = \ol{d}_{\mathtt{h}} (f) + \ol{m}_{\mathtt{h}} (\hat{f}) \cdot \ol{d}_{\hat f(\mathtt{h})} (f^{n-1}) < \ol{d}_\mathtt{h} (f)  + \dfrac{ \ol{m}_\mathtt{h} (\hat{f})}2.$$
When $\hat{f}$ is constant ($\neq \mathtt{h}$) we have that $d_\mathtt{h}(f) > d/2$. In fact, otherwise we would have
$ {d}_{\mathtt{h}}(f^n) = d_\mathtt{h}(f) \cdot d^{n-1} \le  d^n/2$  which would contradict $f \in \cU_n$. Therefore  $d_\mathtt{h}(f)=(d+1)/2$ and $$2 \ol{d}_\mathtt{h} (f) + {\ol{m}_\mathtt{h} (\hat{f})}=\frac{d+1}{d}> 1.$$ 
\end{proof}
\begin{remark}
  The previous lemma suggests that the $n$-unstable set $\cU_n \subset \P^{2d+1}$ has
codimension $d-1$.
Indeed, parametrize maps $f$ in $\ol{\Rat}_d =  \P^{2d+1}$ locally by the location of the zeros $c_1, \dots, c_d$, the poles $p_1, \dots, p_d$ and the value $a$ of $f$ at $z=\infty$.
Then having a hole of depth $D$ corresponds to a union of codimension $2D-1$ 
linear subvarieties of the local parameters 
$(c_1, \dots, c_d, p_1, \dots, p_d, a) \in \C^{2d+1}$.
Having multiplicity $M$ at that hole corresponds to $M-1$ equations on the local parameters.
This suggests that the codimension of having a hole of depth $D$ and multiplicity $M$ 
is $2D+M -2$. Since the codimension of  maps $f$ in $\cU_n$ with constant induced map is large we may assume that $\hat f$ is non-constant. 
In the case that $\hat{f}$ is non-constant and $D < (d+1)/2$, then from Proposition \ref{orbit-hole} we obtain 
 an extra equation for $\cU_n$. Thus the codimension should be at least $2D+M-1$.  
From the above lemma one would conclude that in $\P^{2d+1}$ the codimension of such $f \in \cU_n$ 
 is at least $d $. 
In the case that $D=(d+1)/2$, then $2D+M -2\ge d$ and the codimension is also at least $d$. 
Hence one should expect the dimension of
$\cU_n$  to be at most $d+1$. After projecting $\cU_n$ to moduli space the dimension of $I(\Phi_n)$ should be at most $d-2$.
For $d \ge 3$ and sufficiently large $n$, in
Corollary~\ref{dimension} we exhibit $(d-2)$-dimensional subsets
of $I(\Phi_n)$ formed by conjugacy classes of $n$-unstable maps.
\end{remark}

\subsection{Bad hole of depth $1$}
If the bad hole $\mathtt{h}$ has depth $1$, then  the induced map $\hat f$ is a polynomial or a monomial of degree $d-1$:

\begin{proposition}\label{depth1poly}
Let $d \ge 3$ and consider $f\in\cU_n$ such that $d_{\mathtt{\infty}}(f)=1$ where $\mathtt{h=\infty}$ is the bad hole of $f$. Then,
\begin{enumerate}
\item
$\hat f$ is a degree $d-1$ polynomial or,
\item  $\hat f$ is modulo an affine change of coordinates the
  monomial $z^{-(d-1)}$.
\end{enumerate}
\end{proposition}

\begin{proof}
By Lemma~\ref{depth-multiplicity-inequality},  we have that $m_\mathtt{h} (\hat f) = d-1$. Hence $\deg \hat f = d-1$ and $f$ has a unique hole. It follows that there exists a smallest integer $k\ge 1$ such that $\hat f^k(\infty)=\infty$. For otherwise, $d_{\hat f^\ell(\infty)}(f)=0$  for all $\ell\ge 1$ and  $d_\infty(f^n)=d^{n-1} < d^n /2$ (Lemma \ref{depth-iteration}) but this contradicts the lower bound $d_\infty(f^n)\ge d^n/2$ due to the fact that
$\infty$ is the bad hole (Lemma \ref{bad hole}).
 Thus
$$\frac{d^n}{2}\le d_{\infty}(f^n)\le d^{n-1}+\sum_{i\ge1}d^{n-ik-1}(d-1)^{ik}=\frac{d^{n+k-1}}{d^k-(d-1)^k}.$$
Therefore,
$$d^k-(d-1)^k\le 2d^{k-1}.$$
Hence $k=1$ or $2$. Indeed, if $k\ge 3$, then
$$d^k\left(1-\left(1-\dfrac{1}{d}\right)^k\right)\ge d^k\left(1-\left(1-\dfrac{1}{d}\right)^3\right)=3d^{k-1}-3d^{k-2}+d^{k-3}>2d^{k-1}.$$

If $k=1$, then  $\hat f^{-1}(\infty)=\{\infty\}$, since $m_{\infty}(\hat f)=d-1$. Therefore, $\hat f$ is a degree $d-1$ polynomial.

In the case that $k=2$, we claim $m_{\hat f(\infty)}(\hat f)=d-1$, for otherwise $m_{\hat f(\infty)}(\hat f)\le d-2$. Then we would have

$$d_{\infty}(f^n)\le d^{n-1}+\sum_{i\ge 1}d^{n-1-2i}(d-1)^i(d-2)^i=\frac{d^{n+1}}{d^2-(d-1)(d-2)}<\frac{d^n}{2},$$
which  is a contradiction. Changing coordinates so that $\hat f (\infty) = 0$ we have that  $\hat f (z) = 1/z^{d-1}$ after conjugacy by an appropriate dilation $z \mapsto \lambda z$.
\end{proof}

\section{Non-constant induced map}\label{deep}

As mentioned in the introduction, for $d=2$, Theorem \ref{main} is a consequence of DeMarco's result, see \cite[Theorem 5.1]{DeMarco07}. In this section, we always assume $d\ge 3$. 
Our aim is to prove Theorem~\ref{Thm:perturbation} and the following implication in Theorem \ref{main}:
\begin{theorem}
\label{non-constant}
  If $f \in \nUn$ and $\deg \hat f \ge 1$, then $[f] \in I(\Phi_n)$.
\end{theorem}

\subsection{General strategy}
In this subsection, we state the general strategy to prove Theorem \ref{Thm:perturbation} and Theorem \ref{non-constant}.

From Proposition~\ref{depth1poly} it directly follows that we may organize the proofs of theorems~\ref{Thm:perturbation} and \ref{non-constant} in cases according to the following proposition:

\begin{proposition}
\label{cases}
 Given $f \in \nUn$ such that $\deg \hat f \ge 1$ denote by  $\mathtt{h}$ the bad hole of $f$.
Let $\cO(\mathtt{h})$ be the forward orbit of $\mathtt{h}$ under $\hat f$ and denote by
$\# \cO(\mathtt{h})$ its cardinality.
Then one of the following cases hold:
  \begin{itemize}
\item[Case 0.] $d_{\mathtt{h}}(f) \ge 2$ and  $n \le \# \cO(\mathtt{h})\le \infty$. 
  \item[Case 1.] $d_{\mathtt{h}}(f) \ge 2$, $\# \cO(\mathtt{h}) < n$ and $h$ is strictly preperiodic.
  \item[Case 2.] $d_{\mathtt{h}}(f) \ge 2$,  $\# \cO(\mathtt{h}) < n$ and  $\cO(\mathtt{h})$ is a periodic superattracting orbit.
  \item[Case 3.] $d_{\mathtt{h}}(f) \ge 2$,  $\# \cO(\mathtt{h}) < n$ and  $\cO(\mathtt{h})$ is a periodic but not superattracting orbit.
    \item[Case 4.]  $d_{\mathtt{h}}(f) = 1$ and $\# \cO(\mathtt{h}) =1$.
      \item[Case 5.] $d_{\mathtt{h}}(f) = 1$ and $\# \cO(\mathtt{h}) =2$.
      \end{itemize}
\end{proposition}

With the exception of $d=3$ in cases 4 and 5, we produce
for each $\lambda$ in the complement of a finite subset of $\C$, a degenerate holomorphic family $g_{\lambda,t}$ of degree $d$ such that $g_{\lambda,0} = f$. The construction is implemented so that
for a conveniently chosen holomorphic family of M\"obius transformations $M_t$, we have that
$M_t^{-1} \circ g^n_{\lambda, t} \circ M_t \to G_\lambda$ as $t \to 0$, where $G_\lambda$ is some stable map
of degree $d^n$ (i.e. in $\Rat^{s}_{d^n}$).
Thus, $\Phi_n ([g_{\lambda,t}]) \to [G_\lambda]$ as $t \to 0$.
The construction is also implemented so that
 the GIT-classes $[G_\lambda]$ vary with $\lambda$, which allows us to conclude that $\Phi_n$ has no continuous extension to $[f]$.

The construction site is the Berkovich projective line $\poneberk$. For $d \ge 4$, in Berkovich space language our constructions prove Theorem~\ref{Thm:perturbation}.
We start by regarding the induced map $\hat f$ as an element of $\L(z)$ acting on $\poneberk$ (see Section~\ref{complex-action})
and prescribe a priori the family $M_t$ above to be $$M_t (z) = h_0 + t z.$$
We let $\z_0 = M_t ( \xi_g)$ and proceed to construct  $g_{\lambda,t}$  so that we have control over the
surplus multiplicities of $g^n_{\lambda,t}$ in all the directions $\vec{v} \in T_{\z_0} \poneberk$.

Starting from $\hat f$ there is plenty of flexibility in order to construct $g_{\lambda,t}$ such that $g_{\lambda,0} = f$.
For simplicity, assume that $\infty$ is not a hole of $f$. For each hole $w$ of $f$ of depth $d_w$, we may choose
$2d_w$ arbitrary points $c_1(w), \dots, c_{d_w}(w), p_1(w), \dots, p_{d_w}(w)$  in $B^-_1(w) = \{z \in \L : |z-w| < 1\}$.
Then we let
$$g (z)= \hat f (z) \cdot \prod_{w \in \operatorname{Hole}(f)} \prod_{i=1}^{d_w} \dfrac{z-c_i(w)}{z-p_i(w)} \in \L(z).$$
It follows that the coefficient reduction of $g$ is exactly $f$.
We exploit the flexibility to choose the zeros and poles in $B^-_1(w)$ to construct $g_{\lambda,t}$ controlling the
surplus multiplicities of $g^n_{\lambda,t}$ in all directions at $\z_0$.
Then we apply Lemmas~\ref{depth-surplus} and~\ref{main-lemma} to obtain the depths of the holes of $G_\lambda$, which is the coefficient reduction of $M_t^{-1} \circ g^n_{\lambda, t} \circ M_t$, and use the numerical criteria given by Proposition~\ref{stability-proportional-depth} to certify that $G_\lambda$ is stable. The flexibility of the choices involved will also allow us to verify that $[G_\lambda]$ is not constant (with respect to $\lambda$).

In the exceptional cases 4 and 5 with $d=3$ we were unable to obtain a one parameter family $[G_\lambda]$ as above since the situation turns out to be less flexible, in a certain sense. However, to establish Theorem~\ref{non-constant} we produce two degenerate holomorphic families $f_{t}$ and $g_t$ of degree $d$ rational maps such that $f_0=g_{0} = f$ but $[f_t^n]$ and $[g_t^n]$ converge to distinct elements in $\overline{\mathrm{rat}}_{3^n}$.

The outline of this section is as follows: Section \ref{subsection:notation} contains convenient notations for later use.
Section~\ref{strictly-preperiodic} is devoted to prove that maps which fall into case 1 satisfy the conclusion of Theorem \ref{Thm:perturbation} for any $d\ge 3$ and have GIT-classes in $I(\Phi_n)$. Similarly, in Section~\ref{periodic-superattracting}
 we simultaneously address maps that fall into cases 0 or 2,
and in Section~\ref{periodic-simple} maps in case 3. Cases 4 and 5 are dealt with in Sections~\ref{polynomial} and~\ref{monomial}, respectively.

\subsection{Notation}\label{subsection:notation}

When $f \in \nUn$ is clear from context, we will freely use the following notation.
The bad hole of $f$ will be denoted by $\mathtt{h}$.
For all $j \ge 0$, set
\begin{eqnarray*}
    h_j & =& \hat{f}^j (\mathtt{h}), \\
d_j & =& d_{h_j}(f), \\
    m_j &=& m_{h_j}(\hat{f}).
  \end{eqnarray*}
Thus  the bad hole will be denoted by $\mathtt{h}$ or $h_0$ according to convenience.
  Note that $d_j \le (d+1)/2$ and $m_j \le d-1$ for all $j$.
It will be also convenient to work with the proportional depths and multiplicities:
\begin{eqnarray*}
\ol{d}_j &=& \dfrac{d_j}{d}, \\
  \ol{m}_j &=& \dfrac{m_j}{d}.
\end{eqnarray*}
Thus, $\ol{d}_j \le (d+1)/2d$ and $\ol{m}_j <1$, for all $j$.
The iterated proportional depths and multiplicities are
\begin{eqnarray*}
  \ol{m}_{\mathtt{h}}(\hat{f}^k)  & =& \dfrac{m_{\mathtt{h}}(\hat{f}^k)}{d^k} = \prod_{j=0}^{k-1} \ol{m}_j,\\
  \ol{d}_{\mathtt{h}}(f^k)  & =& \dfrac{d_{\mathtt{h}}(f^k)}{d^k} =
\sum_{j=0}^{k-1} \ol{d}_j \cdot \ol{m}_{\mathtt{h}}(\hat{f}^j).
\end{eqnarray*}

Recall that the threshold proportional depths for (semi)stability are
$$\mu^-(d)=
\begin{cases}
\dfrac{1}{2} & \text{if}\ d\ \text{is even},\\
\\
\dfrac{d-1}{2d} & \text{if}\ d\ \text{is odd},
\end{cases}$$
and
$$\mu^+ (d) = 1 - \mu^- (d) = \begin{cases}
\dfrac{1}{2} & \text{if}\ d\ \text{is even},\\
\\
\dfrac{d+1}{2d} & \text{if}\ d\ \text{is odd}.
\end{cases}$$

Since $f\in\cU_n$, by Proposition \ref{stability-proportional-depth},

\begin{eqnarray*}
  \ol d_{0} (f^n) & > & \mu^+ (d^n), \quad  \text{ if } \hat f^n(\mathtt{h})\neq\mathtt{h},\\
\ol d_{0} (f^n) & \ge & \mu^+ (d^n), \quad  \text{ if } \hat f^n(\mathtt{h}) = \mathtt{h}.\\
\end{eqnarray*}

\subsection{Strictly preperiodic bad hole}
\label{strictly-preperiodic}
In this section, we deal with the maps in Case 1 of Proposition \ref{cases} and we prove that GIT-classes of such maps lie in
$I(\Phi_n)$. More precisely, we prove
\begin{proposition}
\label{strictly-preperiodic-proposition}
  Consider $f \in \nUn$ such that $\deg \hat f \ge 1$ and $d_{\mathtt{h}}(f) \ge 2$.
  Assume that there exist $1 \le q \le n$ and $0 \le \ell \le n-q$ such that $\hat{f}^{\ell+q}(\mathtt{h}) =\hat f^q (\mathtt{h})$. Then Theorem \ref{Thm:perturbation} holds and $[f] \in I(\Phi_n)$.
\end{proposition}

Thus, throughout this subsection we consider $f$ as in the statement of the above proposition.

\medskip
The proof of the proposition is given after we state and prove three lemmas. The first lemma provides us with a lower bound for the total depth along the orbit of $\mathtt{h}$. The second is the construction of $g_{\lambda,t}$ and the third lemma studies the relevant surplus multiplicities.

\begin{lemma}\label{0implies3}
The depths satisfy
 $$\sum_{j=0}^{q+\ell-1} d_j \ge 3.$$
Moreover, $\deg \hat f \ge 2$ and $d\ge 5$.
\end{lemma}

\begin{proof}
By contradiction, suppose that the inequality is false. Then
$d_0=2$ and $d_j=0$ for $1\le j\le q+\ell-1$. The bad hole $h_0$ is strictly preperiodic,
so we would have
$$d_{h_0}(f^n)=d_0d^{n-1} > \frac{d^n+1}{2},$$
which would imply that  $d<4$. Taking into account that $\deg\hat f\ge 1$ and $d_0=2$,
we would have that $d=3$ and  $\deg\hat f=1$. Therefore, $\hat f$ would have no strictly preperiodic points, giving us the desired contradiction.

Since $\hat f$ has a strictly preperiodic point, $\deg \hat f\ge 2$. Then
$$d=\sum_{j=0}^{q+\ell-1} d_j+\deg\hat f\ge 3+2=5.$$
\end{proof}

As previously mentioned, our construction starts regarding $\hat f$ as a rational map in $\L(z)$ acting on $\poneberk$. It is convenient to introduce the relevant
geometric situation in $\poneberk$ before stating the basic properties of our construction, compare with Figure~\ref{fig-123}. Let
\begin{eqnarray*}
  \zeta_0 &:= & \xi_{h_0,|t|},\\
  \zeta_j &:= & \hat{f}^j (\zeta_0), \quad j \ge 1.
\end{eqnarray*}
Observe that $\zeta_j$ lies in the direction at the Gauss point containing
 $h_j$. 
Let $X$ be the convex hull of $\zeta_0, \dots, \zeta_{n-1}$, thus
$$X = \bigcup_{0 \le j \le n-1} [\xi_g,\z_j].$$
Denote by $\vec{w}_\infty \in T_{\z_{0}} \poneberk$ the direction at $\z_0$ 
containing the Gauss point. Let $\vec{w}_z$ the direction at $\z_0$ containing $h_0 + z t$. The directions $\vec{w}_0, \vec{w}_1, \vec{w}_\lambda$ containing
$h_0, h_0+t, h_0 + \lambda t$ will play an important role in our constructions.
For $j=1, \dots, n-1$,  let $\vec{v}_j \in  T_{\z_{j}} \poneberk$ denote the direction containing $h_j$.

\begin{figure}
\center{
\includegraphics{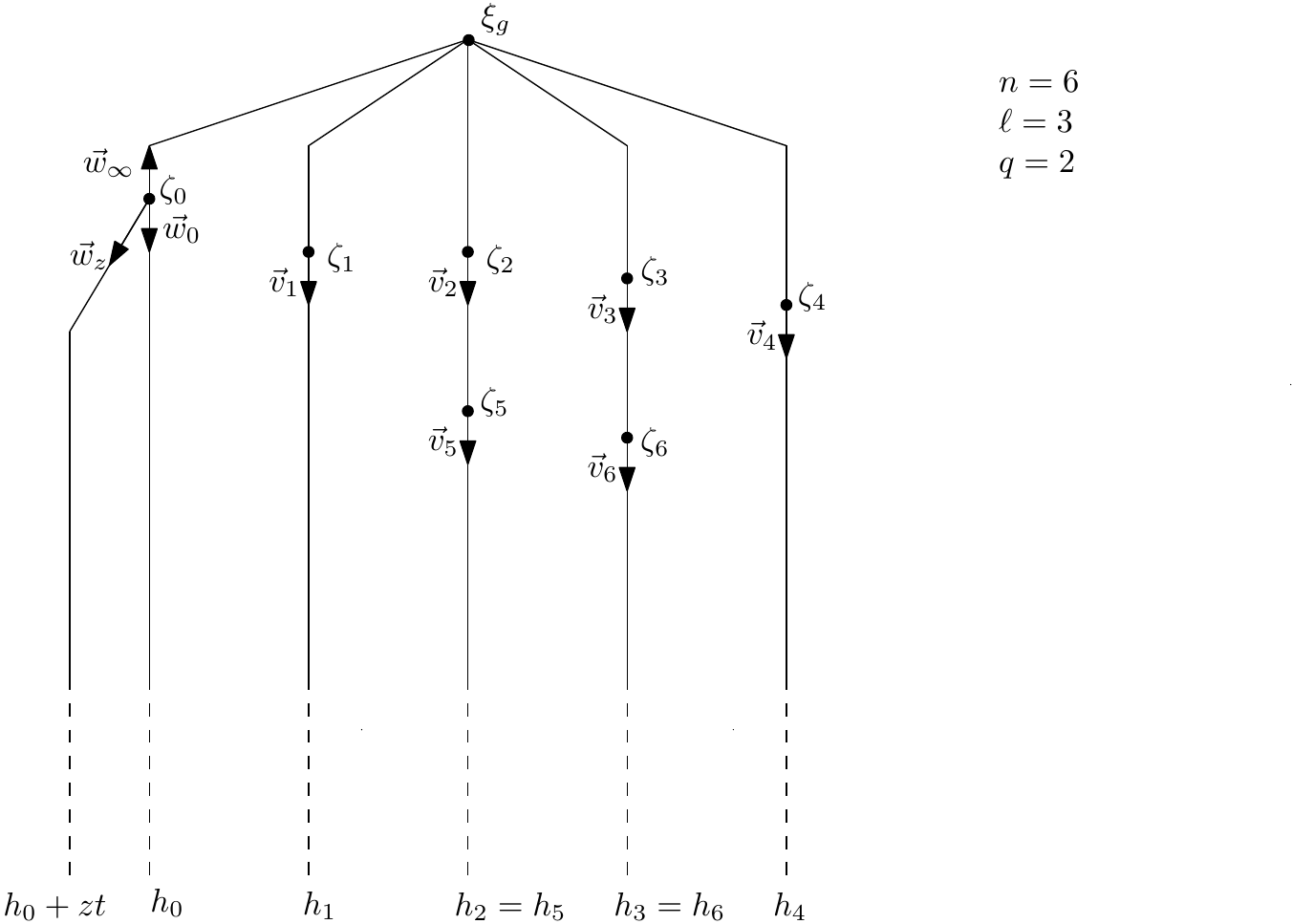}}
\caption{Sketch of geometric situation in $\poneberk$ for Section~\ref{strictly-preperiodic}}
\label{fig-123}
\end{figure}

In the next lemma, for $\lambda \in \C$ outside a finite set we construct  $g_{\lambda, t} (z) \in \C[\lambda, t] (z) \subset \L(z)$. 
The rational map $g_{\lambda, t} (z)$ acting on the Berkovich projective line should be regarded as a ``perturbation'' of $\hat{f}$. 
In $\overline{\Rat}_d$ we will have that $g_{\lambda, t} \to f$ as $t \to 0$. So regarded as elements of $\overline{\Rat}_d$ the maps  $g_{\lambda, t}$ are ``perturbations''
of $f$. The action in $\poneberk$ will agree with that of $\hat{f}$ ``close'' to the Gauss point.  In fact, it agrees with $\hat{f}$ in $X$.
 However, the degree of $\hat{f}$ is increased by conveniently 
adding zeros and poles to $\hat{f}$.  We do so as to ``spread out'' the depth  multiplicity of the bad hole $h_0$ in different directions at $\z_0$ aiming at  
having a stable reduction $g^n_{\lambda,t}$  at the point $\z_0$ (equivalently, $M_t^{-1} \circ g^n_{\lambda,t} \circ M_t \to G_\lambda \in \Rat^s_{d^n}$). We ``put'' surplus multiplicity $1$ in several directions at $\z_0$ so that under iterations these directions do not fall into directions with positive surplus multiplicities. However, the directions $\vec{w}_0$ and $\vec{w}_\infty$ will increase their surplus multiplicity under iterations but our construction is so that we obtain the necessary upper bounds for the iterated surplus multiplicity.

Given a direction $\vec{v}$ at some point in $\poneberk$ we denote by
$s_\lambda(\vec{v})$ the surplus multiplicity of $g_{\lambda,t}$ in that direction.
See Figure~\ref{fig-preperiodic-family} for a sketch of the points and directions involved in the construction of $g_{\lambda,t}$ given in the next lemma.

\begin{lemma}
  \label{preperiodic-family}
There exists $g_{\lambda, t} (z) \in \C[\lambda, t] (z) \subset \L(z)$
of degree $d$ such that for all $\lambda$ in the complement of a finite subset of $\C$,
the following statements hold: 
\begin{enumerate}

\item The coefficient reduction $g_{\lambda,0}$ of $g_{\lambda,t}$ is $f$.

\item For all $\xi \in X$,
    \begin{eqnarray*}
      g_{\lambda,t}(\xi) &=& \hat{f}(\xi),\\
      T_\xi g_{\lambda,t} & = & T_\xi \hat{f}.
    \end{eqnarray*}
\item There exist pairwise distinct $c_2, \dots, c_{d_0-1} \in \C \setminus \{ 0, 1 \}$
such that for $\vec{w} \in T_{\z_{0}} \poneberk \setminus \{ \vec{w}_\infty \}$,
$$ s_\lambda(\vec{w}) =
\begin{cases}
  1 & \text{ if }  \vec{w}\in\{\vec{w}_{1}, \vec{w}_\lambda,
\vec{w}_{c_2}, \dots, \vec{w}_{c_{d_0-1}}\}, \\
  0 & \text{ otherwise.}
\end{cases}$$

\item $s_\lambda (\vec{v_j})=d_j$ for $j=1, \dots, n-1$.
\end{enumerate}
\end{lemma}

\begin{proof}
Pick $N$ sufficiently large such that the point $\xi_{h_n,|t|^N}$ lies in the segment $]\zeta_n, h_n[ \subset \poneberk$. For $1\le j<q+\ell$, set $\xi_{j}=\xi_{h_j,|t|^N}$ and $h_i^+=h_i-t^N$.
Define
$$\gamma(z)=\prod_{j=1}^{q+\ell-1}\left(\frac{z-h_i^+}{z-h_i}\right)^{d_j}=\prod_{j=1}^{q+\ell-1}\left(1+\frac{t^N}{z-h_i}\right)^{d_j}.$$
Let $z_1, \dots, z_m \in \P^1$ be the holes of $f$ outside the set $\{h_0, \dots, h_{q+\ell-1}\}$ with corresponding depths $\delta_1, \dots, \delta_m$. We may assume that $z_j \in \L$ for all $j$. Let
$$\beta(z) = \prod_{j=1}^m \left(1+\dfrac{t^N}{z-z_j} \right)^{\delta_j}.$$
Now choose pairwise distinct $c_2,\cdots,c_{d_0-1}\in \C\setminus\{0, 1\}$ and set
$$\alpha(z)=\prod_{j=2}^{d_0-1}\left(1+\frac{t^N}{z-(h_0+c_it)}\right).$$
Consider
$$g_{\lambda,t}(z)=\left(1+\frac{t^N}{z-(h_0+t)}\right)\left(1+\frac{t^N}{z-(h_0+\lambda t)}\right)\alpha(z)\beta(z)\gamma(z)\hat f(z).$$
For all $\lambda$ distinct from  $0,1, c_2, \dots, c_{d_0-1}$ statement (1) follows from the formula of $g_{\lambda,t}(z)$.
Statements (2)--(5) are a direct consequence of  Corollary~\ref{perturbation-c}.
\end{proof}

\begin{figure}
\includegraphics{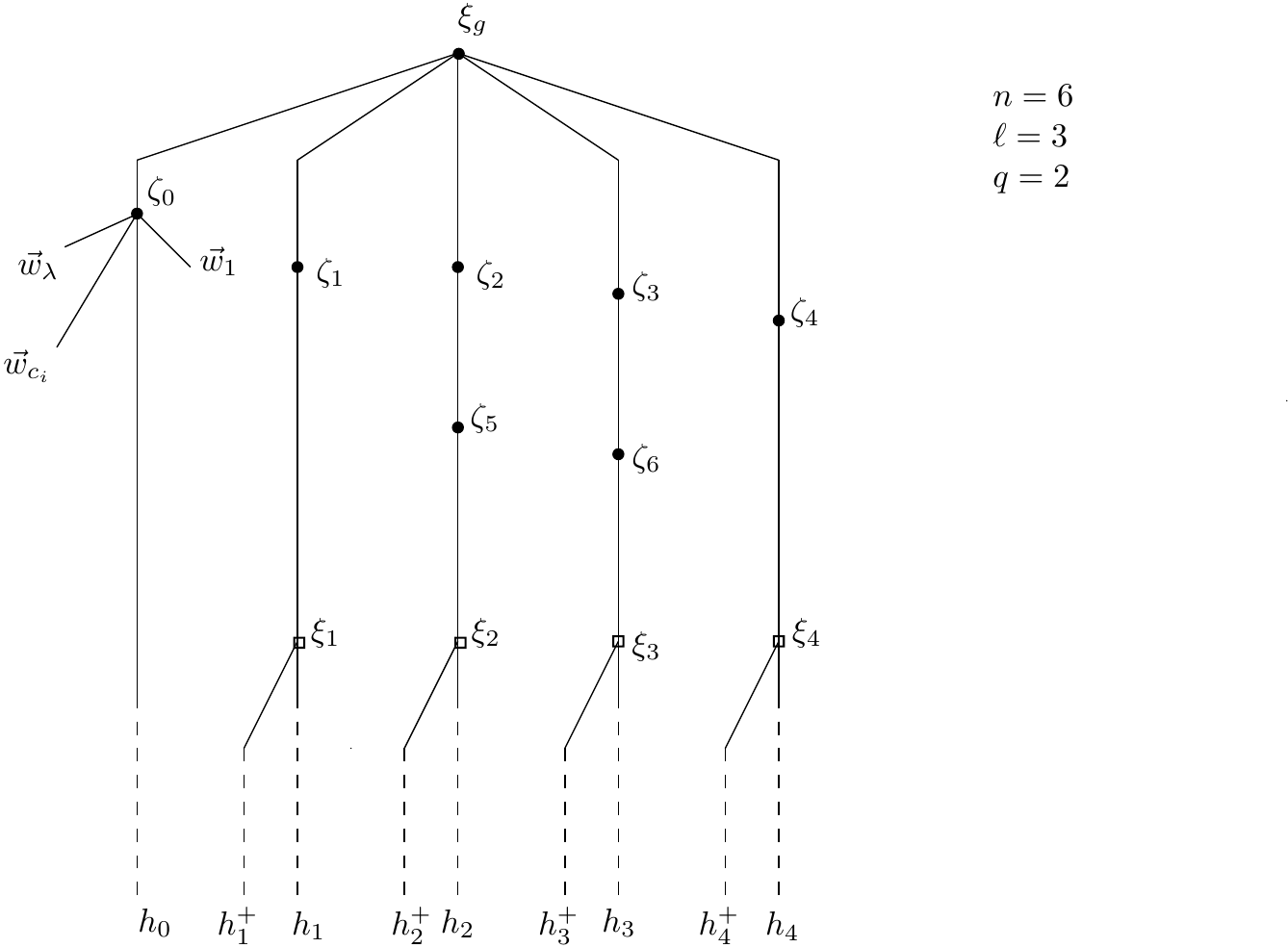}
\caption{Sketch of points and directions involved in Lemma~\ref{preperiodic-family}.}
\label{fig-preperiodic-family}
\end{figure}

Given $\z \in \poneberk$ and a direction $\vec{v} \in T_\z \poneberk$, for a map $g_{\lambda,t}$
as in the previous lemma, we will consistently denote by ${s}^j_\lambda(\vec{v})$ the surplus multiplicity of
$g^j_{\lambda,t}$ in the direction $\vec{v}$ and by ${m}^j_\lambda(\vec{v})$ the corresponding multiplicity.
For $j=1$, we omit the superscript. Similarly,  $\ol{s}^j_\lambda(\vec{v})$ and  $\ol{m}^j_\lambda(\vec{v})$
denote the corresponding proportional multiplicities.
Our aim is to control $s^n_\lambda(\vec{w})$ for $\vec{w} \in T_{\z_0} \poneberk$.

\begin{lemma}\label{preperiodic-surplus}
  Let $g_{\lambda,t}$ be such that (1)--(5) of Lemma~\ref{preperiodic-family} hold.
  Then for all but finitely many $\lambda \in \C$
  the following statements
  also hold:
 \begin{enumerate}
\item If $\vec{w} \in T_{\z_0} \poneberk$ is distinct from $\vec{w}_\lambda, \vec{w}_0, \vec{w}_1, \vec{w}_{c_2} \dots, \vec{w}_{c_{d_0-1}}, \vec{w}_\infty$, then 
$${s}^n_\lambda (\vec{w}) =0.$$
  \item $$\ol{s}_\lambda^n(\vec{w}_0)  \le {\mu^-(d^n)}.$$
  Moreover, if $d_{j_0}\ge 1$ for some $1\le j_{0}\le q+\ell-1$, then
  $$0<\ol{s}_\lambda^n(\vec{w}_0).$$
\item If $\vec{w} = \vec{w}_1, \vec{w}_\lambda$ or $\vec{w}_{c_j}$ for some $j=2, \dots, d_0-1$,  $$ \ol{s}^n_\lambda (\vec{w}) = \dfrac{1}{d.}$$
\item $$\ol{m}_\lambda^n (\vec{w}_\infty) + \ol{s}_\lambda^n(\vec{w}_\infty) < \mu^-(d^n).$$
\end{enumerate}
\end{lemma}
\begin{proof}
By construction, in directions $\vec{w}$ at $\z_0$ distinct from $\vec{w}_\lambda, \vec{w}_0, \vec{w}_1, \vec{w}_{c_2}, \dots, \vec{w}_{c_{d_0-1}}, \vec{w}_\infty$ the map  $g_{\lambda,t}$ has zero surplus multiplicity.
For $j \ge 1$, at $\z_j$ the only directions that may have positive surplus multiplicities are $\vec{v}_j$ and the direction of the Gauss point.
For all $\vec{w} \in T_{\z_0} \poneberk$ we have that $T_{\z_0} g^j_{\lambda,t} (\vec{w})$ agrees with $T_{\z_0} \hat{f}^j_{\lambda,t} (\vec{w})$.
In view of Section~\ref{complex-action}  we know that  $T_{\z_0} \hat{f}^j_{\lambda,t} (\vec{w})$ 
is the direction $\vec{v}_j$ or the direction of the Gauss point at $\z_j$ if and only if 
$\vec{w} = \vec{w}_0$ or $\vec{w}=\vec{w}_\infty$, respectively.
 Thus, $s^n_\lambda (\vec{w}) =0$ for all directions $\vec{w}$ at $\z_0$ distinct from $\vec{w}_\lambda, \vec{w}_0, \vec{w}_1, \vec{w}_{c_2}, \dots, \vec{w}_{c_{d_0-1}}, \vec{w}_\infty$. Hence we have proven statement (1). Moreover, for $\vec{w} = \vec{w}_1, \vec{w}_\lambda$ or $\vec{w}_{c_j}$ for some $j=2, \dots, d_0-1$, we have that 
$s^n_\lambda (\vec{w}) = d^{n-1} s_\lambda (\vec{w}) = d^{n-1}$ and statement (3) also follows.

For statement (2), we apply the formula for  $\ol{s}_\lambda^n(\vec{w}_0)$ given by
Lemma \ref{surplus-iterate} taking into account that $s_\lambda(\vec{w}_0)=0$
and that proportional multiplicities are bounded above by $1$ to obtain
$$\ol{s}_\lambda^n(\vec{w}_0) =\ol{m}_\lambda (\vec{w}_0)
\ol{s}_\lambda^{n-1}({\vec{v}_1} ) \le
\ol{s}_\lambda^{n-1}({\vec{v}_1} ),$$
since $\vec{v}_1=T_{\z_0}g_{\lambda,t}(\vec{w}_0)$.

Note
 $\ol{s}_\lambda^{n-1}({\vec{v}_1} )\le \mu^-(d^{n})$. For otherwise, we would have that $\ol{s}_\lambda^{n}({\vec{v}_1} )  >\ol{s}_\lambda^{n-1}({\vec{v}_1} ))>  \mu^-(d^n)$. Since  $\ol{s}_\lambda^{n}({\vec{v}_1} )$ may be written as a rational number with  denominator $d^n$, it follows that $\ol{s}_\lambda^{n}({\vec{v}_1} ) \ge 1/2$ and hence $h_1$ would be a bad hole. By the uniqueness of the bad hole
(Lemma \ref{bad hole}), we would conclude that
 $h_0=h_1$,  which is a contradiction with the strict preperiodicity of $h_0$. Thus $\ol{s}_\lambda^n(\vec{w}_0) \le \mu^-(d^n)$.

Moreover, it follows from Lemma \ref{surplus-iterate} that $0<\ol{s}_\lambda^n(\vec{w}_0)$ since by Proposition \ref{preperiodic-family} (2) and (5) we have $\ol s_\lambda (T_{\zeta_0}(g_{\lambda,t}^{j_0}(\vec{w}_0))=d_{j_0}/d>0$. Hence statement (2) holds.

In order to prove statement (4),
recall that
$$\deg_{\z_0} g_{\lambda,t}^n  + \sum_{\vec{w} \in T_{\z_0} \poneberk }s^n_\lambda(\vec{w})=d^n,$$
(see Equation~(\ref{surplus-sum}) in Section~\ref{berkovich}).
Since $T_{\z_j} g_{\lambda,t} = T_{\z_j} \hat f$, we have
 $\deg_{\z_0} g_{\lambda,t}^n = m_0 \cdots m_{n-1} = m^n_\lambda(\vec{w}_\infty)$. Therefore,
$$\ol{m}^n_\lambda(\vec{w}_\infty) + \ol{s}^n_\lambda(\vec{w}_\infty) =
1 - \sum_{\substack{\vec{w}\not=\vec{w}_\infty \\ \vec{w} \in T_{\z_0} \poneberk }}\ol{s}^n_\lambda(\vec{v}).$$
Now
\begin{eqnarray*}
  \sum_{\substack{\vec{w}\not=\vec{w}_\infty \\ \vec{w} \in T_{\z_0} \poneberk} }\ol{s}^n_\lambda(\vec{v}) &=&
\frac{d_0}{d} + \ol{s}^n_\lambda(\vec{w}_0)\\
&= &{\ol{d}_0} +  \sum_{j=1}^{n-1} \ol{m}_\lambda^j(\vec{w}_0)\ol{s}_\lambda(T_{\z_0}g^j_{\lambda,t}(\vec{w}_0))\\
&= &{\ol{d}_0}+ \sum_{j=1}^{n-1} \ol{m}_0^j \ol{d}_j =  \ol d_0(f^n).
\end{eqnarray*}
By hypothesis $f \in \cU_n$ and  $\hat f^n (h_0) \neq h_0$, it follows that
$$\mu^{+}(d^n) <  \ol d_0(f^n),$$
Thus
$$\ol{m}^n_\lambda(\vec{w}_\infty) + \ol{s}^n_\lambda(w_\infty) < 1-\mu^+(d^n) = \mu^{-} (d^n),$$
and statement (4) follows.
\end{proof}

\begin{proof}[Proof of Proposition~\ref{strictly-preperiodic-proposition}]
For $\lambda$ in the complement of the finite set where the previous lemmas
hold, we let
$$G_\lambda(z) = \lim_{t \to 0} M^{-1}_t \circ g^n_{\lambda,t}\circ M_t (z),$$
where $M_t (z) = h_0 + t z$. Note that if we regard $M_t$ as a degree $1$ rational map
in $\L(z)$, we conclude that the coefficient reduction of $M^{-1}_t \circ g^n_{\lambda,t}\circ M_t$
is $G_\lambda$. 
The direction $\vec{w}$ at $\z_n$ that contains $\z_0$ is the direction containing the Gauss point. Since $\vec{w}_\infty$ is the 
unique direction at $\z_0$ which maps onto 
$\vec{w}$ under $T_{\z_o} g_{\lambda,t}^n$, we may apply Corollary~\ref{relative-position} and the previous lemma to conclude that the proportional depths of all the holes of $G_\lambda$ are bounded above
by $\mu^-(d^n)$. Then $G_\lambda \in \Rat^s_{d^n}$ according to Proposition~\ref{stability-proportional-depth}.

It only remains to show $[G_\lambda]$ is not constant in $\lambda$.
If there exists $1\le j_0\le q+\ell-1$ such that $d_{j_0}\ge 1$, by Lemma \ref{preperiodic-surplus} we have $\{0, \infty, 1, \lambda\}\subset\mathrm{Hole}(G_\lambda)$. If $d_j=0$ for all $j\ge 1$, then $d_0\ge 3$ and $\{c_2, \infty, 1, \lambda\}\subset\mathrm{Hole}(G_\lambda)$. In both cases,
$\mathrm{Hole}(G_\lambda)$ has at least $4$ elements including $\infty, 1$ and $ \lambda$ and we claim that there exists $\lambda_0\not=\lambda_1$  such that $[G_{\lambda_0}] \neq [G_{\lambda_1}]$. Indeed, the list of cross ratios of the holes of $G_\lambda$ cannot be independent of $\lambda$.  For otherwise, they would be uniformly bounded away from $0$ and $\infty$. However, when $\lambda$ approaches $1$ or $\infty$,  at least one cross ratio approaches $0$ or $\infty$.
Hence there are non-conjugate choices for  $G_\lambda$. Therefore, the construction is such that Theorem \ref{Thm:perturbation} holds and hence $[f]\in I(\Phi_n)$.
\end{proof}

\subsection{Periodic superattracting or large bad hole orbit}
\label{periodic-superattracting}
In this subsection we show that for  $n$-unstable maps $f$ which fall into cases 0 and 2 of Proposition~\ref{cases}, we have that $[f] \in I(\Phi_n)$.

\begin{proposition}
\label{depth>1}
  Given $n \ge 2$ assume that $f\in\cU_n$ with non-constant induced map $\hat{f}$ and the bad hole $\mathtt{h}$ such that $d_{\mathtt{h}}(f) \ge 2$.
Assume that $\# \cO (\mathtt{h}) \ge n$ or $\cO (\mathtt{h})$ is a periodic superattracting orbit.
Then Theorem \ref{Thm:perturbation} holds and $[f] \in I(\Phi_n)$.
\end{proposition}

Observe that in this case $\cO (\mathtt{h})$ can be infinite. The hard case is when the orbit is in fact periodic of period
less than $n$. However, since the same construction applies for $\# \cO (\mathtt{h}) \ge n$, we deal simultaneously with both situations. In a first reading we suggest to assume that $\cO(\mathtt{h})$ is periodic of small period compared to $n$.

Throughout this subsection we consider $f$ as in the statement of the above proposition. 
As before we regard $\hat{f}$ as an element of $ \L (z)$ which acts on $\poneberk$
and proceed to construct $g_{\lambda,t}$. The points and directions involved
in the construction are illustrated in Figure~\ref{fig-periodic-super}.
The holomorphic families $g_{\lambda,t}$ are obtained in the first lemma further below.
These holomorphic families $g_{\lambda,t}$  depend on the choice of some integers which will be adjusted in
the proof of the proposition at the end of the subsection.

Let us start labeling some points and directions in our construction site $\poneberk$, see Figure~\ref{fig-periodic-super}.
Let
$$\ell = \min \{ p , n : \hat{f}^p (\mathtt{h}) = \mathtt{h} \}.$$
That is, $\ell =n$ unless $\mathtt{h}$ is periodic of period $p < n$.
Without loss of generality we assume that $h_j \neq \infty$ for all $j$.

Here we are going to consider $M_t^{-1} \circ g_{\lambda,t} \circ M_t$ for $M_t (z) = h_0 +  t^2 z$.
Thus the relevant point $\z_0$ in $\poneberk$ is given by $\z_0 = M_t (\xi_g)=\xi_{0,|t|^2}$ and its forward orbit is
$$\z_j = \hat{f}^j (\z_0 ) = \xi_{h_j, |t|^{2 {m}_{\mathtt{h}}(\hat{f}^j)}}.$$

All the construction depends on an integer $k_\star$ with $0< k_\star <n$. This integer $k_\star$ will be adjusted later
so that the reduction of $g^n_{\lambda,t}$ at $\z_0$ is stable for a generic value of $\lambda$ (equivalently, the coefficient reduction of $ M_t^{-1} \circ g_{\lambda,t} \circ M_t$ is stable.)
Given an integer $k_\star$ such that $0< k_\star <n$, apply the division algorithm to write
$$k_\star = q_\star \ell + r_\star,$$
where $q_\star \ge 0$ and  $0 \le r_\star < \ell$.

The idea again is to spread the surplus multiplicities along the orbit of $\z_0$ to obtain good bounds for the surplus multiplicities in all directions at $\z_0$. The directions that are more difficult to 
control are $\vec{v}_0, \vec{v}_\infty \in T_{\z_0} \poneberk$ which denote the directions of $h_0$ and the Gauss point, respectively. Intuitively, the integer $k_\star$ is going to be related to the iterate so that the surplus multiplicity in the direction $\vec{v}_0$ will stop to increase. However, one pays the cost that in the direction $\vec{v}_\infty$ the surplus multiplicity will start to increase faster after the corresponding iterate. Achieving the perfect balance is the key of the construction.


To spread the surplus multiplicities appropriately, we focus on the iterates between $\ell q_\star$ and $\ell (q_\star +1) -1$ of $\z_0$ and introduce a pair of Berkovich type II points $\xi^\pm_r$ above and below $\z_{\ell q_\star + r}$, where $r=0,\dots, \ell-1$, as follows. 
Let $$\mu_r = {m}_{\mathtt{h}}(\hat{f}^{\ell q_\star + r}) $$ and
$$
\xi^\pm_r =
  \xi_{h_r,|t|^{{2\mu_r} \pm 1}}.
$$
At each $\xi^\pm_r$ choose a direction $\vec{u}^\pm_r$ not containing $h_r$ nor the Gauss point.

\begin{figure}
\includegraphics{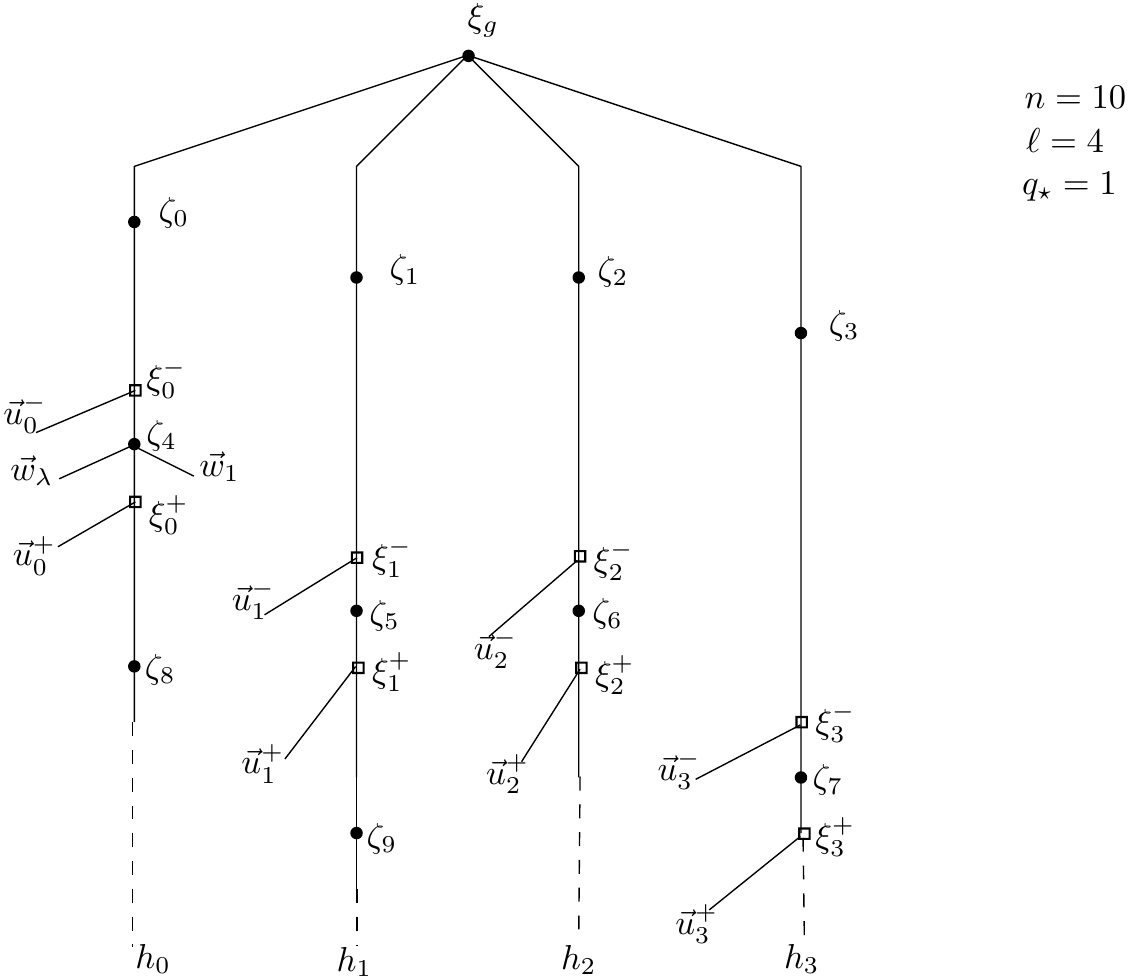}
\caption{Sketch of points and directions involved in the proof of Proposition~\ref{depth>1}.}
\label{fig-periodic-super}
\end{figure}

Let us intuitively describe some aspects of our construction. For each $r=0, \dots, \ell-1$ we may place $d_r$ zeros and poles in the direction of $h_j$ at the Gauss point.
To achieve the required balance, for $r > 0$, we put all of them in the directions $\vec{u}_r^\pm$. We  do so
as to have the highest possible surplus multiplicity  in $\vec{u}_r^+$ for $0<r < r_\star$ (here $r_\star$ corresponds to the iterate $k_\star$). That is, we put all the available surplus multiplicity (i.e. $d_r$) in a direction which is below $\z_{q_\star \ell + r}$ but above $\z_{(q_\star +1) \ell + r}$. In turn for $r > r_\star$ we put the available multiplicity above $\z_{q_\star \ell + r}$ in the direction $\vec{u}_r^-$. For $r=r_\star$ (i.e. around $\z_{k_\star}$) we put some of surplus multiplicity above and some below in a proportion that will be adjusted later  in order to achieve the aforementioned balance. For $r=0$, we also put  multiplicity in  directions at  $\z_{q_\star \ell}$.  
The precise properties of our construction are stated in the lemma below including how we spread multiplicities for $r=0$.

As in the previous section we denote by  $s_\lambda (\vec{v})$  the surplus multiplicity of $g_{\lambda,t}$ in the direction $\vec{v}$. Also it is convenient to let $X$ be the convex hull of the points $\z_j$ for $j=0, \dots, \max\{(q_\star+1) \ell -1,n\}$. That is,
$$X = \bigcup_{j=0 }^{n'} [\xi_g,\z_j]$$
where $n'=\max\{(q_\star +1)\ell -1,n\}$.

In a first reading  we suggest to suppose $r_\star \neq 0$ (i.e. $\ell$ does not divide $k_\star$).

\begin{lemma}\label{superattacting-case}
Let  $d_{r_\star}^\pm \ge 0$ be integers such that
$$d_{r_\star}^+ + d_{r_\star}^- =d_{r_\star}$$
 and if $r_\star =0$, then $d_{0}^+ \ge 2$.

There exists $g_{\lambda, t} (z) \in \C[\lambda,t] (z) \subset \L(z)$
of degree $d$ such that for $\lambda$ in the complement of finite set in $\C$, the following statements hold:

\begin{enumerate}
\item
The coefficient reduction $g_{\lambda,0}$ of $g_{\lambda,t}$ is $f$.
\item For all $\xi \in X$,
    \begin{eqnarray*}
      g_{\lambda,t}(\xi) &=& \hat{f}(\xi),\\
      T_\xi g_{\lambda,t} & = & T_\xi \hat{f}.
    \end{eqnarray*}

\item In $T_{\z_{\ell q_\star}} \poneberk$, let
 $\vec{w}_0$ be the direction  of $h_0$ and  $\vec{w}_\infty$ be the direction
of the Gauss point.
There exists two directions $\vec{w}_1$ and $\vec{w}_\lambda$
 with surplus multiplicities $1$ such that  the cross ratio $[\vec{w}_0,  \vec{w}_1, \vec{w}_\lambda, \vec{w}_\infty] = \lambda$.

\item If $r_\star \neq 0$, then
\begin{eqnarray*}
 \label{eq:1}
s_\lambda (\vec{u}^+_0) &=& d_0-2,\\
s_\lambda (\vec{u}^+_r) &=& d_r, \,\, \mbox{ for } \,\,\, r=1, \dots, r_\star-1,\\
s_\lambda (\vec{u}^\pm_{r_\star}) &=&d_{r_\star}^\pm, \\
s_\lambda (\vec{u}^-_r) &=& d_r, \,\, \mbox{ for } \,\,\, r=r_\star +1, \dots, \ell-1.
\end{eqnarray*}

\item If $r_\star =0$, then
\begin{eqnarray*}
  \label{eq:2}
  s_\lambda (\vec{u}^-_{0}) &=& d_0^-,\\
  s_\lambda (\vec{u}^+_{0}) &=& d_0^+-2,\\
  s_\lambda (\vec{u}^-_r) & =& d_r,  \,\, \mbox{ for } \,\,\, r=1, \dots, \ell-1.
\end{eqnarray*}
\end{enumerate}
\end{lemma}

\begin{proof}
  Let $N$ be a sufficiently large integer.
  For each $r =0, \dots, \ell-1$, choose $u^{\pm}_r \in \C\setminus \{0\}$ such that
$p_r^\pm =h_r+u_r^\pm t^{2 \mu_r \pm 1}$ is in the direction $\vec{u}_r^\pm$ at $\xi_r^\pm$ and
consider the degree $1$ map with a pole at $p_r^\pm$ and a zero at $p_r^\pm - t^N$:
$$\gamma_r^{\pm} (z) = 1+\dfrac{t^N}{z-p_r^\pm}.$$
  Let $z_1, \dots, z_m \in \P^1$ be the holes of $f$ outside the set $\{h_0, \dots, h_{\ell-1}\}$ with corresponding depths $\delta_1, \dots, \delta_m$. We may assume that $z_j \in \C \subset \L$ for all $j$. Let
$$\beta(z) = \prod_{j=1}^m \left(1+\dfrac{t^N}{z-z_j} \right)^{\delta_j}.$$

Now consider
   $$\alpha_\lambda (z) = \left(1+\dfrac{t^N}{z- (h_0+ \lambda t^{\mu_0})} \right) \cdot
 \left(1+\dfrac{t^N}{z- (h_0+  t^{\mu_0})} \right).$$

If  $r_\star = 0$, then let
$$g_{\lambda,t} (z) = \hat{f}(z) \cdot \alpha_\lambda (z) \cdot \beta(z) \cdot
\left(\gamma_0^{+} (z) \right)^{d^+_0-2} \cdot \left(\gamma_0^{-} (z) \right)^{d^-_0}
 \prod_{r=1}^{\ell-1} \left(\gamma_r^{-} (z) \right)^{d_r}. $$

If $r_\star \neq 0$, then let
$$g_{\lambda,t} (z) =  \hat{f}(z) \cdot \alpha_\lambda (z) \cdot \beta(z) \cdot
\left(\gamma_0^{+} (z) \right)^{d_0-2} \cdot
 \left(\gamma_{r_\star}^{+} (z) \right)^{d^+_{r_\star}} \cdot
 \left(\gamma_{r_\star}^{-} (z) \right)^{d^-_{r_\star}}
\prod_{r=1}^{r_\star-1}
\left(\gamma_r^{+} (z) \right)^{d_r} \prod_{r=r_\star+1}^{\ell-1} \left(\gamma_r^{-} (z) \right)^{d_r}. $$

Statement (1) follows from the above formulas for $g_{\lambda,t}$. For $N$ sufficiently large,
Lemma \ref{construction} guarantees that statements (2)--(5) hold.
\end{proof}

By construction if $\vec{v} \in T_{\z_0} \poneberk$ is a direction for which $s_\lambda^n (\vec{v}) >0$,
then $\vec{v}$ maps in $q_\star \ell$ iterates onto $\vec{w}_1$ or $\vec{w}_\lambda$, $\vec{v} =\vec{v}_0$, or $\vec{v} = \vec{v}_\infty$. We analyze the surplus multiplicities $s_\lambda^n (\vec{v})$ in our next lemma for $\vec{v} \neq \vec{v}_\infty$. Our control of 
$s_\lambda^n (\vec{v}_\infty)$ will follow from equation \eqref{surplus-sum}

\begin{lemma}
Let  $g_{\lambda,t}$ be as above. Then the following statements hold:
  \begin{enumerate}
  \item 
$$\ol{s}_\lambda^n(\vec{v}_0) = \left(\sum_{j=0}^{k_\star-1} \ol{d}_j \cdot \ol{m}_{\mathtt{h}}(\hat{f}^j) \right)+
\dfrac{d_{r_\star}^+}{d} \cdot \ol{m}_{\mathtt{h}}(\hat{f}^{k_\star}) -
\dfrac{2}{d} \cdot \ol{m}_{\mathtt{h}}(\hat{f}^{q_\star \ell}).$$
 \item There are $2 m_\mathtt{h}(\hat f^{q_\star\ell})$ directions $\vec{v} \in T_{\z_0} \poneberk$ such that $T_{\z_0} \hat{f}^{q_\star \ell} (\vec{v}) = \vec{w}_1$ or $\vec{w}_\lambda$,
  \item If $\vec{v} \in T_{\z_0} \poneberk$ is such that $T_{\z_0} g_{\lambda,t}^{q_\star \ell} (\vec{v}) = \vec{w}_1$ or $\vec{w}_\lambda$, then
$$\ol{s}^n_\lambda (\vec{v}) = \frac{1}{d^n}.$$
  \end{enumerate}
\end{lemma}

\begin{proof}
Since $T_{\z_j}g_{\lambda,t} = T_{\z_j} \hat f$, for $0\le j\le n-1$,  the direction $T_{\z_0}g_{\lambda,t}^j(\vec{v}_0)$ is the direction containing $h_j$.
Therefore,
$$ m_{\lambda}(T_{\z_0}g_{\lambda,t}^j(\vec{v}_0)) = m_j$$
and
$$ s_{\lambda}(T_{\z_0}g_{\lambda,t}^j(\vec{v}_0))=
\begin{cases}
d_j & 0\le j\le q_\star\ell-1,\\
d_0-2 & j=q_\star\ell, \\
d_j & q_\star\ell+1\le j\le k_\star-1,\\
d_{k_\star}^+ & j=k_\star,\\
0 & j\ge k_\star.
\end{cases}$$
Statement (1) now follows from Lemma \ref{surplus-iterate}. Note $T_{\z_0} \hat{f}^{q_\star \ell} $ has degree $m_\mathtt{h}(\hat f^{q_\star\ell})$ and neither $\vec{w}_1$ nor $\vec{w}_\lambda$ is a critical value of  $T_{\z_0}g_{\lambda,t}^{q_\star \ell} $. Hence statement (2) holds.  Let $\vec{v} \in T_{\z_0} \poneberk$ satisfying $T_{\z_0} g_{\lambda,t}^{q_\star \ell} (\vec{v}) = \vec{w}_1$ or $\vec{w}_\lambda$,
$$ s_{\lambda}(T_{\z_0}g_{\lambda,t}^j(\vec{v}))=
\begin{cases}
0& 0\le j < q_\star\ell,\\
1& j=q_\star\ell, \\
0 & q_\star\ell+1\le j\le n-1.
\end{cases}$$
Thus statement (3) holds.
\end{proof}

\begin{proof}[Proof of Proposition~\ref{depth>1}]
Now we have to adjust $k_\star$ and $d^\pm_{r_\star}$. We choose $k_\star$ so that the first term in the previous lemma's formula for $\bar{s}_\lambda^n(\vec{v}_0)$ is as large as allowed in order to have stable reduction for $g_{\lambda,t}^n$ at $\z_0$.
That is,
for $1\le i\le n$, define
$$\mu_i= \sum_{j=0}^{i-1} \ol{d}_j \cdot \ol{m}_j.$$
Then $\mu_i$ is nondecreasing. Let $k_\star \ge 1$ be the largest integer such that if $\hat f^n(h_0)\neq h_0$, then $\mu_{k_\star}\le \mu^+(d^n)$; if $\hat f^n(h_0)=h_0$, then $\mu_{k_\star}<\mu^+(d^n)$. Note $\ol d_{h_0}(f^n)=\mu_n$. By Proposition \ref{stability-proportional-depth}, such a $k_\star<n$ exists.

Write $k_\star = \ell q_\star + r_\star$ where $q_\star \ge 0$ and $ 0 \le r_\star < \ell$. Now it is time to choose $d_{r_\star}^+$. Again the idea is to choose it as large as stable reduction allows.
More precisely, choose $ d_{r_\star}^+$ with $0 \le d_{r_\star}^+ \le d_{r_\star}$ and such that if $\hat f^n(h_0)\neq h_0$, then
$$
\sum_{j=0}^{k_\star-1} \ol{d}_j \cdot \ol{m}_j +
\dfrac{d_{r_\star}^+-2}{d} \cdot\ol m_{k_\star}\le\mu^-(d^n)\le\mu^+(d^n)< \sum_{j=0}^{k_\star-1} \ol{d}_j \cdot \ol{m}_j + \dfrac{d_{r_\star}^+}{d} \cdot\ol m_{k_\star};$$
if $\hat f^n(h_0)=h_0$, then
$$
\sum_{j=0}^{k_\star-1} \ol{d}_j \cdot \ol{m}_j +
\dfrac{d_{r_\star}^+-2}{d} \cdot\ol m_{k_\star}<\mu^-(d^n)\le\mu^+(d^n)\le \sum_{j=0}^{k_\star-1} \ol{d}_j \cdot \ol{m}_j + \dfrac{d_{r_\star}^+}{d} \cdot\ol m_{k_\star}.
$$
Moreover, when $r_\star=0$, we choose $d_{r_\star}^+ \ge 2$.

Now let $g_{\lambda, t}$ be the family given by the previous lemmas associated to the above choices
of $k_\star$ and $d_{r_\star}^+$. To check that $g_{\lambda,t}^n$ has stable reduction at $\z_0$ it is convenient to define
$$\Delta:=\sum_{j=0}^{k_\star-1} \ol{d}_j \cdot \ol{m}_j + \dfrac{d_{r_\star}^+}{d} \cdot\ol m_{k_\star}.$$
Note that
$$\ol{s}^n_\lambda (\vec{v}_0) = \Delta - \dfrac{2}{d} \cdot \ol{m}_{q_\star \ell}.$$
We claim that if $\hat f^n(h_0)\neq h_0$,
$$\Delta - \dfrac{2}{d} \cdot \ol{m}_{q_\star \ell}\le \mu^-(d^n);$$ 
if $\hat f^n(h_0)= h_0$,
$$\Delta - \dfrac{2}{d} \cdot \ol{m}_{q_\star \ell}< \mu^-(d^n).$$
Indeed, $$\dfrac{1}{d} \cdot \ol{m}_{k_\star} = \dfrac{1}{d} \cdot \ol{m}_{q_\star \ell} \cdot \ol{m}_{r_\star} \le  \dfrac{1}{d} \cdot \ol{m}_{q_\star \ell}.$$
It follows that if $\hat f^n(h_0)\neq h_0$, we have $\ol{s}^n_\lambda (\vec{v}_0)\le\mu^-(d^n)$; if $\hat f^n(h_0)= h_0$, we have $\ol{s}^n_\lambda (\vec{v}_0)<\mu^-(d^n)$.

Now we proceed to find an upper bound for $\ol{s}^n_\lambda (\vec{v}_\infty)$.
Since $g_{\lambda,t}^n (\z_0) \neq \z_0$, we have that
$$ \sum_{\vec{v} \in T_{\z_0} \poneberk} \ol{s}^n_\lambda (\vec{v}) = 1.$$
Moreover,
$$\sum_{\substack{\vec{v} \in T_{\z_0} \poneberk\\ \vec{v} \neq \vec{v}_0, \vec{v}_\infty}}
\ol{s}^n_\lambda (\vec{v})  = \dfrac{2}{d} \cdot \ol{m}_{q_\star \ell}.$$
Thus $$\ol{s}^n_\lambda (\vec{v}_\infty) =1 - \Delta.$$
It follows that if $\hat f^n(h_0)\neq h_0$, we have $\ol{s}^n_\lambda (\vec{v}_\infty)<\mu^-(d^n)$; if $\hat f^n(h_0)=h_0$, we have $\ol{s}^n_\lambda (\vec{v}_\infty)\le\mu^-(d^n).$

After change of coordinates we may assume that $h_0 =0$. For all but finitely many $\lambda \in \C$, we let
  $$G_\lambda (z) = \lim_{t \to 0} \dfrac{g_{\lambda,t}^n (t^2z)}{t^2} \in \ol{\Rat}_{d^n}.$$
  Then if $\hat f^n(h_0)\neq h_0$, the induced map $\widehat G_\lambda=[1:0]\in\P^1$; if $\hat f^n(h_0)= h_0$, the induced map $\widehat G_\lambda=[0:1]\in\P^1$. By Proposition \ref{stability-proportional-depth}, it follows that $G_\lambda$ is stable. Thus in moduli space,
$$\Phi_n([g_{\lambda,t}]) = [g_{\lambda,t}^n] \to [G_\lambda] \in \ol{\rat}_{d^n}$$
while
$$[g_{\lambda,t}] \to [f] \in \ol{\rat}_d.$$
The holes of $G_\lambda$ are at $0$, $\infty$ and the preimage under $z \mapsto z^{m_\mathtt{h}(\hat f^{q_\star\ell})}$ of $1$ and $\lambda$.
Hence the cross ratios of the holes vary with $\lambda$. For otherwise, these cross ratios would be bounded away from $0$ and $\infty$, which is clearly not the case when $\lambda$
converges to $0, 1,$ or $ \infty$.
Thus, $[G_\lambda]$ is not constant on $\lambda$. The construction is such that Theorem \ref{Thm:perturbation} holds and it follows that $[f] \in I(\Phi_n)$.
\end{proof}

\subsection{Periodic but not superattracting bad hole orbit}
\label{periodic-simple}
Now we deal with the maps in Case 3 of Proposition \ref{cases}. Our goal is to prove
\begin{proposition}
\label{simple}
  Given $n \ge 2$ assume that $f\in\cU_n$ has non-constant induced map $\hat{f}$ and the bad hole $\mathtt{h}$ is such that $d_{\mathtt{h}} (f) \ge 2$. If $\mathtt{h}$
  has a critical point free periodic orbit under $\hat{f}$ of period $\ell < n$,
  then Theorem \ref{Thm:perturbation} holds and $[f] \in I(\Phi_n)$.
\end{proposition}

First we prove that Proposition \ref{simple} holds under the assumption $$\sum_{j=0}^{\ell-1} d_j \ge 3.$$
Afterwards we consider the exceptional case in which $h_0 = \mathtt{h}$ is the only hole in its orbit and $h_0$ has depth exactly $2$.

Our construction of $g_{\lambda,t}$ will be so that at the point $\z_0 = \xi_{h_0,|t|}$ it has stable reduction, for $\lambda$ in the complement of a finite set.
Let $$\z_j = \hat{f}^j(\z_0).$$
As before, let 
$$X = \bigcup_{j=1}^{\ell-1} [\xi_g,\z_j].$$
In $T_{\z_{0}} \poneberk$, let
 $\vec{w}_0$  be the direction of $h_0$ and let $\vec{w}_\infty$ be the direction
of the Gauss point.

\subsubsection{Proof of Proposition~\ref{simple}: the generic case}
We consider  $f \in\cU_n$ as in the statement of the previous proposition and
assume that $$\sum_{j=0}^{\ell-1} d_j \ge 3.$$
Without loss of generality we also assume that $h_j \neq \infty$ for all $j$. Note that since $m_{h_j} (\hat{f}) =1$ for all $j$, we have that $\z_j = \xi_{h_j,|t|}$.

\begin{lemma}
\label{construction4a}
  There exists $g_{\lambda, t} (z) \in \C[\lambda, t] (z) \subset \L(z)$
of degree $d$ such that for all $\lambda$ in the complement of a finite subset of $\C$, the following statements hold:

\begin{enumerate}
\item
The coefficient reduction $g_{\lambda,0}$ of $g_{\lambda,t}$ is $f$.
\item For all $\xi \in X$,
    \begin{eqnarray*}
      g_{\lambda,t}(\xi) &=& \hat{f}(\xi),\\
      T_\xi g_{\lambda,t} & = & T_\xi \hat{f}.
    \end{eqnarray*}
\item
In $T_{\z_{0}} \poneberk$, there exists two  directions $\vec{w}_1$ and $\vec{w}_\lambda$, each
 with surplus multiplicity $1$, such that the cross ratio $[\vec{w}_0, \vec{w}_1, \vec{w}_\lambda,  \vec{w}_\infty] = \lambda$.
\item For all directions $\vec{v}$ in $T_{\z_{j}} \poneberk$ not containing
the Gauss point, 
 $s_\lambda(\vec{v}) \le 1$.
\item
There exists $0\le j\le \ell-1$ such that the direction $T_{\z_{0}}g_{\lambda,t}^j(\vec{w}_0)$ has nonzero surplus multiplicity.
\end{enumerate}
\end{lemma}
\begin{proof}
  We work with subscripts mod $\ell$ so that $\hat f (\z_j) = \z_{j+1}$ for all $j=0, \dots, \ell-1$.
  In  coordinates for $T_{\z_j} \poneberk \equiv \C \cup \{ \infty \}$ where $\infty$ corresponds to the direction of the Gauss point, the map $T_{\z_j} \hat f$ is affine.
  For $0\le j\le \ell-1$ and $1\le i\le d_j$, we choose complex numbers $c_i^{(j)}$ and denote
  by $v_i^{(j)}$ the direction in $T_{\z_j} \poneberk$ containing $h_j + c_i^{(j)} t$. Our construction will be so that $v_i^{(j)}$ has surplus multiplicity $1$ for all $i$ and $j$. Our choice is such that the following hold:
  \begin{itemize}
  \item $c_{1}^{(0)}=1$. This will guarantee that in the direction $\vec{w}_1$ at $\z_0$ that contains $h_0 + t$ the surplus multiplicity is $1$.
  \item If  $d_0\ge 3$, then $c_{2}^{(0)}=0$. The objective of this choice is that if we have a sufficiently deep bad hole $h_0$, then we will have surplus multiplicity $1$ in the direction $\vec{w}_0$ at $\z_0$.
  \item If $d_0 =2$ and $j_0$ is the smallest  $j \ge 1$ such that $d_j \neq 0$, then $c_{1}^{(j_0)}=0$. The idea here is that if the depth of $h_0$ is small, then the direction $\vec{w}_0$ maps in $j_0$ iterates onto the direction of $h_{j_0}$ at $\z_{j_0}$ which will have surplus multiplicity $1$.
  \item For all $j=0, \dots, \ell-1$ and all $ 1 \le i < k \le d_j$, we have $v_i^{(j)} \neq v_{i'}^{(j)}$.
\end{itemize}
Remark that to  apply Lemma \ref{construction} it will be sufficient to just take $N=2$.

Now set
$$\alpha_0(z)=\prod_{i=1}^{d_0-1}\left(1+\frac{t^2}{z-(h_0+c^{(0)}_{i}t)}\right),$$
and for $1\le j\le\ell-1$, set
$$\alpha_j(z)=\prod_{i=1}^{d_j}\left(1+\frac{t^2}{z-(h_j+c^{(j)}_{i}t)}\right).$$
Let $z_1, \dots, z_m \in \P^1$ be the holes of $f$ outside the set $\{h_0, \dots, h_{\ell-1}\}$ with corresponding depths $\delta_1, \dots, \delta_m$. We may assume that $z_j \in \C \subset \L$ for all $j$ and let
$$\beta(z) = \prod_{j=1}^m \left( 1+\dfrac{t^2}{z-z_j} \right)^{\delta_j}.$$
Define
$$g_{\lambda,t}(z)=\left(1+\frac{t^2}{z-(h_0+\lambda t)}\right)\hat f(z)\beta(z)\alpha_0(z)\prod_{j=1}^{\ell-1}\alpha_j(z).$$
Statement (1) follows the formula for $g_{\lambda,t}$. Statement (2) follows Lemma \ref{construction}. Taking $\lambda$ outside the finite set of $\C \setminus \{ 0,1\}$ for which $\vec{w}_\lambda$ agrees with $v_0^{(j)}$ for some $1 \le j \le d_0$ we have that statement (3) holds by construction and since $c_{1}^{(0)}=1$.
By our choice of $c^{(j)}_i$, for any direction $\vec{v}\in T_{\z_j}\poneberk$ not containing the Gauss point, we have that $s_\lambda(\vec{v}) \le 1$. That is, statement (4) holds.
If $d_0\ge 3$, then the direction $\vec{w}_0$ has surplus multiplicity $1$ since $c_{2}^{(0)}=0$.
If $d_0=2$, then $T_{\z_{0}}g_{\lambda,t}^{j_0}(\vec{w}_0)$ is $\vec{v}_{1}^{(j_0)}$ which is also a direction with surplus multiplicity $1$ since $c_{1}^{(j_0)}=0$. Therefore, statement (5) holds.
\end{proof}

\begin{lemma}
\label{surplus4a}
Let $g_{\lambda,t} \in \C[\lambda,t](z) \subset \L(z)$ be as in the previous lemma. Then for all $\lambda \in \C$ in the complement of a finite set, the following statements hold:
\begin{enumerate}
\item If $\vec{w}\in T_{\z_0}\poneberk$ is a direction not containing the Gauss point, then
$$\ol s_\lambda^n(\vec{w}) < \mu^-(d^n).$$
\item $$\ol s^n_\lambda(\vec{w}_\infty)=1  - \dfrac{1}{d^n} - \ol{d_0}(f^n).$$
\end{enumerate}
\end{lemma}
\begin{proof}
Since $$1 \le \deg \hat f \le d - \sum_{j=0}^{\ell-1} d_j \le d -3,$$
we have that $d \ge 4$.

For any $\vec{w}\in T_{\z_0}\poneberk$ not containing the Gauss point, we have $s_\lambda(T_{\z_0} \hat f^j (\vec{w}))\le 1$ for all $0 \le j \le n-1$.
Therefore,
$$\ol{s}_\lambda^n (\vec{w}) = \sum_{j=0}^{n-1} \dfrac{1}{d^j} \cdot \ol{s}_\lambda(T_{\z_0} \hat f^j (\vec{w}))  \le  \sum_{j=0}^{n-1} \dfrac{1}{d^{j+1}} = \dfrac{1}{d-1} \cdot \left(1 - \dfrac{1}{d^n} \right) < \mu^{-}(d^n),$$
since $d \ge 4$.

Using that $$\ol s_\lambda (T_{\z_0} \hat g_{t,\lambda}^j (\vec{w}_\infty)) = 1 - \dfrac{1}{d} - \dfrac{d_j}{d},$$
we have  $$\ol s_\lambda^n(\vec{w}_\infty) = \sum_{j=0}^{n-1} \dfrac{1}{d^j} \ol s_\lambda (T_{\z_0} \hat g_{t,\lambda}^j (\vec{w}_\infty)) = 1  - \dfrac{1}{d^n} - \sum_{j=0}^{n-1} \dfrac{d_j}{d^{j+1}}
= 1  - \dfrac{1}{d^n} -\ol{d_0}(f^n). $$
\end{proof}

Now we finish the proof of  Proposition~\ref{simple}
under the assumption that
$$\sum_{j=0}^{\ell-1} d_j \ge 3.$$

For $\lambda$ in the complement of the finite set where the previous lemmas
hold, we let  $$G_\lambda(z) = \lim_{t \to 0} M_t^{-1}\circ g^n_{\lambda,t}\circ M_t (z)$$
where $M_t (z) = h_0 + t z$. As in the previous cases, we conclude that the coefficient reduction of $M_t \circ g_{\lambda,t}^n \circ M^{-1}_t$
is $G_\lambda$.

By Corollary~\ref{relative-position}  and the previous lemma, we have that $\ol{d}_z (G_\lambda) < \mu^- (d^n)$ for all $z \neq \infty$. Moreover,
$$\ol d_{\infty}(G_\lambda)=
\begin{cases}
\ol s^n(\vec{w}_\infty)+\dfrac{1}{d^n} &\ \text{if}\ \hat f^n(h_0)\neq h_0,\\
\ol s^n(\vec{w}_\infty) &\ \text{if}\  \hat f^n(h_0)=h_0.
\end{cases}
$$
Taking into account that $f\in\cU_n$,
we have that if $\hat f^n(h_0)\neq h_0$, then $\ol{d_0}(f^n) > \mu^+ (d^n)$; if $\hat f^n(h_0) =  h_0$, then $\ol{d_0}(f^n) \ge \mu^+ (d^n)$. Therefore,
$$\ol{d}_\infty (G_\lambda) < \mu^+ (d^n),$$
and we conclude that $G_\lambda$ is stable.

The maps $g_{\lambda,t}$ were constructed so that for all but finitely many $\lambda$,
$$\{0,1,\infty,\lambda\}\subset\mathrm{Hole}(G_{\lambda}).$$
Therefore, the list of cross ratios
of the holes of $G_\lambda$ is not constant with respect to $\lambda$
and $[G_\lambda] \in \ol{\rat}_{d^n}$ is non-constant. Thus, Theorem \ref{Thm:perturbation} holds and $[f]$ lies in $I(\Phi_n)$. \hfill $\Box$

\subsubsection{Proof of Proposition~\ref{simple}: the exceptional cases}
Here we assume that $f\in\cU_n$ is a map as in the statement of the proposition such
that $\sum d_j< 3$.
Note that if $d \ge 5$, by Lemma ~\ref{depth-multiplicity-inequality}, we have $d_0 \ge 3$. Hence, $d = 3$ or $4$, $d_0=2$  and $d_1= \cdots=d_{\ell-1}=0$.
 Since $f$ is semistable, $h_1 \neq h_0$.
Thus we may assume that $h_0 =0$ and $h_1 = \infty$.
\begin{lemma}
\label{construction4b}
 There exists $g_{\lambda, t} (z) \in \C[\lambda, t] (z) \subset \L(z)$
of degree $d$ such that for all $\lambda$ in the complement of a finite subset of $\C$, the following statements hold:
\begin{enumerate}
\item The coefficient reduction $g_{\lambda,0}$ of $g_{\lambda,t}$ is $f$.

\item For all $\xi \in X$,
    \begin{eqnarray*}
      g_{\lambda,t}(\xi) &=& \hat{f}(\xi).\\
      T_\xi g_{\lambda,t} & = & T_\xi \hat{f}, \quad \text{ if } \xi \neq \z_j.
    \end{eqnarray*}

\item $\deg_{\z_0} g_{\lambda,t} = 2$ and $\deg_{\z_j} g_{\lambda,t} =1$ for all $j=1, \dots, \ell-1$.
  Moreover, $T_{\z_k} g_{\lambda, t}$ is independent of $\lambda$, for all $0 \le k \le \ell-1$.

\item If $\vec{w}$ is a direction at $\z_0$ such that $m_\lambda (\vec{w}) > 1$, then
  $T_{\z_0} g_{\lambda, t}^j (\vec{w})$ is not the direction containing the Gauss point at $\z_j$ for all $j=1, \dots, n$.

\item There exist directions $\vec{u}_0, \vec{w}_1 \in T_{\z_0} \poneberk$
independent of $\lambda$ such that $\vec{w}_1 \neq \vec{w}_\infty$, $T_{\z_0} g_{\lambda,t} (\vec{w}_1)$
is the direction containing the Gauss point at $\z_1$, and $T_{\z_0} g^\ell_{\lambda,t} (\vec{u}_0) =\vec{w}_1$.

\item There exists a direction $\vec{u}_\lambda \neq \vec{w}_{\infty}$ in $T_{\z_0} \poneberk$ with non-zero surplus multiplicity. Moreover, $s_\lambda ( \vec{u}_\lambda) = 1$. Furthermore, the cross ratio $[ \vec{u}_0, \vec{w}_1, \vec{u}_\lambda, \vec{w}_\infty] = \lambda$.
\end{enumerate}
\end{lemma}

\begin{proof}
 Observe that
$\zeta_1 = \hat{f} (\z_0) =\xi_{0, |t|^{-1}} $ and for $k = 2, \dots, \ell-1$ we have that $\z_k = \hat{f}^k (\zeta_0) = \xi_{h_k, |t|}$.  For $k \neq 1$, we consider the
coordinate of $T_{\z_k} \poneberk$ that identifies $w \in \P^1 \equiv \C \cup \{\infty \}$ with
the direction containing $h_k + w t$. In $T_{\z_1} \poneberk$,  the coordinate $w \in \P^1$ corresponds to the direction containing $t/w$. For $0 \le k \le \ell-1$,  in these coordinates, $T_{\z_k} \hat{f} (w) = a_k w$ for some $0 \neq a_k \in \C$.

We construct  $g_{\lambda,t}$ in two steps.
Given $\beta \in \C \setminus \{0\}$ to be chosen later,
first we consider
$$\phi_{\beta,t} (z) = \dfrac{z-t+\beta^2 t}{z-t} \cdot \hat{f}(z),$$
and select a convenient $\beta$. Our selection of $\beta$ is so that the assertions corresponding to statements (1)--(4) hold for $\phi_{\beta,t}$. Although statements (1)--(3) are rather straightforward, statement (4) is more subtle. That is, we have to show that critical directions at $\z_0$ are not eventually mapped onto the direction of the Gauss point. The existence of an appropriate parameter $\beta$ will be obtained by analyzing the situation when $\beta$ is arbitrarily small.
In fact, a direct computation shows that $\phi_{\beta,t} (\xi) = \hat{f} (\xi)$ for all
$\xi \in [\xi_g,\z_j]$ and all $0 \le j \le \ell-1$. Moreover,
\begin{eqnarray*}
  T_{\z_0} \phi_{\beta, t} (w) & = & a_0 w \cdot \dfrac{w -1 + \beta^2}{w-1}, \\
  T_{\z_1} \phi_{\beta, t} (w) & = & a_1 w, \\
  T_{\z_k} \phi_{\beta, t} (w) & = &  a_k w + \dfrac{h_{k+1}}{h_k} \beta^2, \quad \text{ for \,\,} 2 \le k \le \ell-1. \\
\end{eqnarray*}
Let $A=\prod\limits_{j=0}^{\ell-1}a_j$. It follows that there exists $C \in \C$ such that
$$q_\beta(w) := T_{\z_0} \phi^\ell_{\beta, t} (w) =
A w  \cdot \dfrac{w -1 + \beta^2}{w-1} +C \beta^2 = Aw + \beta^2 \left(
\dfrac{A w}{w-1} +C \right).$$
The directions with multiplicity $2$ at $\z_0$ under $\phi_{\beta,t}$ correspond to the critical points $w = 1 \pm \beta$ of $T_{\z_0} \phi_{\beta, t}$ which are also the critical points of  $q_\beta$.
We claim that there exists $\beta$ such that the directions corresponding to  $w = 1 \pm \beta$ do not map
to the direction containing the Gauss point (i.e.,  $w=\infty$) under $T_{\z_0} \phi_{\beta,t}^k$ for all $k=1, \dots, n$.
For otherwise, there exists $p_\pm$ such that $p_\pm \ell <n$ and
$$q_\beta^{p_\pm} (1 \pm \beta) = 1$$
for all $\beta$. In particular, this occurs for $\beta$ arbitrarily close to $0$, and therefore  $A^{p} =1$ for some  $p$ dividing $p_\pm$, since
$q^{p_\pm}_\beta(1 \pm \beta)= A^{p_\pm} + O(\beta)$. Assuming $p$ is the smallest number such that  $A^{p} =1$, it follows that
$$q_{\beta}^{p} (1 + w \beta) = 1  + \left( w + \dfrac{1}{w} \right) \beta + o(\beta).$$
Let $$P(w) = w + \dfrac{1}{w}$$
and observe that the critical points $\pm 1$ of $P$  have infinite forward orbit converging to the (double) parabolic point at infinity. Therefore, $P^m (\pm 1) \neq \infty$ for all $m$. 
Since $P(0) = \infty$, also $P^m (\pm 1)  \neq 0$.
Thus, given $k$,
$$q_{\beta}^{kp} (1 \pm  \beta) = 1 + P^k (\pm 1) \beta + o(\beta) \neq 1$$
for $0 \neq \beta$ sufficiently close to $0$.
Hence we may choose $\hat \beta$ such that $q_{\hat \beta}^{kp} (1 \pm \hat\beta) \neq 1$ for all $k$ such that $k\ell < n$.

Once chosen $\hat \beta$ we continue with the second step of the construction of $g_{\lambda,t}$.
Let $z_1, \dots, z_m \in \P^1$ be the holes of $f$ which are not $h_0$ with corresponding depths $\delta_1, \dots, \delta_m$. We may assume that $z_j \in \L$ for all $j$. Let
$$\gamma(z) = \prod_{j=1}^m \left(1+\dfrac{t^N}{z-z_j} \right)^{\delta_j}.$$
Let $u_0 \in \C $ be such that $q_{\hat{\beta}} (u_0) =1$. Note that $T_{\z_0} \phi_{\hat{\beta},t}^{\ell+1} (u_0) = \infty$.
Given $\lambda \in \C \setminus \{0,1\}$, let $a(\lambda)$ be the affine function such
that the cross ratio $[ u_0, 1, a(\lambda), \infty] =\lambda$.
Now we can introduce $g_{\lambda,t} \in \C[\lambda,t](z)$ as
$$g_{\lambda,t} (z) = \dfrac{z - (a(\lambda) t - t^N)}{z - (a(\lambda) t+t^N)}\gamma(z)\phi_{\hat\beta,t} (z),$$
where $N \in \N$ is sufficiently large.

Statement (1) follows from the above formula.
Note that for $N$ large, $g_{\lambda,t} (\xi) = \phi_{\hat \beta, t} (\xi) = \hat f (\xi)$ and  $T_{\xi} g_{\lambda,t} = T_{\xi} \phi_{\hat \beta, t}$ for all $\xi \in [\xi_g,\z_k]$ and all $k$. Therefore, statements (2) and (3) hold.
Statement (4) follows from our careful choice of $\hat \beta$. Taking $\vec{u}_0$ (resp. $\vec{w}_1, \vec{u}_\lambda$) as the direction corresponding to $w=u_0$
(resp. $w=1, w=a(\lambda)$), statements (5) and (6) follow.
\end{proof}

\begin{lemma}
  Let $g_{\lambda,t}$ be such that (1)--(6) of Lemma~\ref{construction4b} hold.
Then  for all $\lambda$ in the complement of a finite subset of $\C$, the following statements hold:
\begin{enumerate}
\item $$\ol{s}_\lambda (\vec{w}_\infty) = 1 - \dfrac{1}{d^n} - \ol{d}_0 (f^n).$$
\item If $\vec{w} \in T_{\z_0} \poneberk$ is such that $T_{\z_0} g_{\lambda,t}^k (\vec{w})$ contains the Gauss point for some $0 \le k < n$, then
$$0 \neq \ol{s}_\lambda^n (\vec{w}) \leq  \ol{s}^n_\lambda (\vec{w}_\infty).$$
\item If $\vec{w} \in T_{\z_0} \poneberk$ is such that $T_{\z_0} g_{\lambda,t}^k (\vec{w}) =\vec{u}_\lambda$  for some $k$, then
 $$\ol{s}_\lambda^n (\vec{w}) = \dfrac{1}{d^{k+1}}.$$
\item If $\vec{w} \in T_{\z_0} \poneberk$ is such that $T_{\z_0} g_{\lambda,t}^k (\vec{w})$
is distinct from $\vec{u}_\lambda$ and does not contain the Gauss point, for all $k=0, \dots, n-1$, then $\ol{s}_\lambda^n (\vec{w}) =0$.
\end{enumerate}
\end{lemma}

\begin{proof}
Consider the subset $\Lambda$ of $\C \setminus \{0,1\}$ such that for all $\lambda \in
\Lambda$ the following statements hold:
\begin{itemize}
\item[(i)] $\vec{w}_1 \neq T_{\z_0} g_{\lambda,t}^{k\ell}(\vec{u}_\lambda)$ for all $k$ such that $k \ell < n$.
 \item[(ii)]  $\vec{u}_\lambda \neq T_{\z_0} g_{\lambda,t}^{k \ell}(\vec{u}_\lambda)$ and $m_\lambda(T_{\z_0} g_{\lambda,t}^{k \ell}(\vec{u}_\lambda))=1$ for all $k$ such that $k \ell < n$.
\item[(iii)] If $w\in T_{\z_0} \poneberk$ is such that $m_\lambda(\vec{w})=2$, then  $T_{\z_0} g_{\lambda,t}^{k \ell}(\vec{w}) \neq \vec{u}_\lambda $ for all $k$ such that $k \ell < n$.
\end{itemize}

Since $T_{\z_0} g_{\lambda,t}^\ell$ is a rational self-map of $T_{\z_0} \poneberk \equiv \P^1$ independent of $\lambda \in \C \setminus \{0,1\}$ and conditions (i)--(iii) are violated for
finitely many $\vec{u}_\lambda \in  T_{\z_0} \poneberk$, we conclude that $\Lambda$ is the complement of a finite subset of $\C$.
Let $\lambda \in \Lambda$.
By Formula~(\ref{surplus-sum}), for all $j=0, \dots, n-1$,
$$\ol s_\lambda (T_{\z_0}  g_{t,\lambda}^j (\vec{w}_\infty)) = 1 - \dfrac{1}{d} - \dfrac{d_j}{d}.$$
As in the proof of Lemma~\ref{surplus4a}, it follows that statement (1) holds.
Statement (4) is a direct consequence of Lemma~\ref{construction4b} (6) together with the fact that for $j=1,\dots,{\ell-1}$, any direction at $\z_j$ not containing the Gauss point
has zero surplus multiplicity, since $d_j =0$.

For statement (2), we consider $\vec{w} \in T_{\z_0} \poneberk$ such that $T_{\z_0} g_{\lambda,t}^k (\vec{v})$ contains the Gauss point for some $0 \le k < n$. We may assume that $k$ is the minimal
iterate with this property.  Therefore, from (i) above and Lemma~\ref{construction4b} (6),
$$\ol{s}_\lambda^n (\vec{w}) = \sum^{n-1}_{j=0}
\ol{m}_\lambda^j(\vec{w}) \ol{s}_\lambda (T_{\z_0} g_{\lambda,t}^j (\vec{w})) =
  \sum^{n-1}_{j=k}
\ol{m}_\lambda^j(\vec{w}) \ol{s}_\lambda (T_{\z_0} g_{\lambda,t}^j (\vec{w}_\infty)).$$
Moreover, from statement (4) of the previous lemma, $\ol m_\lambda^j (\vec{w}) = d^{-j}$ and hence
 $$\ol{s}_\lambda^n (\vec{w}) =
  \sum^{n-1}_{j=k}
d^{-j} \ol{s}_\lambda (T_{\z_0} g_{\lambda,t}^j (\vec{w}_\infty)) \le \sum^{n-1}_{j=0}
d^{-j} \ol{s}_\lambda (T_{\z_0} g_{\lambda,t}^j (\vec{w}_\infty))=\ol{s}_\lambda^n (\vec{w}_\infty).$$
That is, statement (2) holds.

Finally, if $\vec{w} \in T_{\z_0} \poneberk$ is such that $T_{\z_0} g_{\lambda,t}^k (\vec{w}) =\vec{u}_\lambda$  for some $k$, then by (i), (ii) and Lemma~\ref{construction4b} (6),
$$\ol{s}_\lambda^n (\vec{w}) =
\ol{m}_\lambda^k(\vec{w}) \ol{s}_\lambda (\vec{u}_\lambda),$$
 and by (iii) we have $\ol{m}_\lambda^k(\vec{w}) = d^{-k}$. Thus statement (3) holds.
\end{proof}

Now we finish the proof of  Proposition~\ref{simple}
under the assumption that
$$\sum_{j=0}^{\ell-1} d_j < 3.$$
We will only consider $\lambda$ for which the previous lemmas hold and
let $$G_\lambda (z) = \lim_{t \to 0} \dfrac{g^n_{\lambda,t}(tz)}{t}.$$

We claim that $G_\lambda$ is stable.
The holes of $G_\lambda$ correspond to directions $\vec{w}$ at $\z_0$ such that one of the following holds:
\begin{enumerate}
\item There exists  $0 \le k \le n$ such that $T_{\z_0} g_{\lambda,t}^k (\vec{w})$ is the direction
at $\z_k$ containing the Gauss point.
\item There exists $0 \le k < n$ such that $\ell$ divides $k$ and
$T_{\z_0} g_{\lambda,t}^k (\vec{w})$ is the direction containing $a(\lambda) t$ at $\z_0$.
\end{enumerate}
Since $f\in\cU_n$, if $\hat f^n(h_0)\neq h_0$, then
$\ol{d}_0(f^n)>\mu^+(d^n)$; if $\hat f^n(h_0)=h_0$, then $\ol{d}_0(f^n)\ge\mu^+(d^n)$.
Therefore, if $\vec{w}\in T_{\z_0}\poneberk$ and $k$ is minimal so that (1) holds, then
$$\ol{s}_\lambda^n(\vec{w}) \le \ol{s}_\lambda^n (\vec{w}_\infty) = 1 - \dfrac{1}{d^n} - \ol{d}_0 (f^n).$$
If $\hat f^n(h_0)\neq h_0$, then
$$\ol{s}_\lambda^n(\vec{w}) \le \ol{s}_\lambda^n (\vec{w}_\infty)<\mu^-(d^n)- \frac{1}{d^n};$$
if $\hat f^n(h_0)=h_0$, then
$$\ol{s}_\lambda^n(\vec{w}) \le \ol{s}_\lambda^n (\vec{w}_\infty)\le \mu^-(d^n)-\frac{1}{d^n}.$$
Now let $\vec{w}$  be a  direction at $\z_0$ and $k$ minimal such that (2) holds.
Then
$$\ol{s}_\lambda^n(\vec{w}) = \dfrac{1}{d} \cdot \dfrac{1}{d^k} < \mu^-(d^n).$$
Now consider a hole $z$ of $G_\lambda$. Then the direction $\vec{w}$ at $\z_0$ containing $zt$ satisfies either (1) or (2).
In case (1), by Corollary~\ref{relative-position}, since $\ol{m}^n(\vec{w})= d^{-n}$,
$$\ol{d}_z(G_\lambda) =
\begin{cases}
\ol{s}_\lambda^n (\vec{v}) + \dfrac{1}{d^n} & \text{if}\ \hat f^{n}(h_0)\not=h_0,\\
\ol{s}_\lambda^n (\vec{v}) &  \text{if}\  \hat f^{n}(h_0)=h_0.
\end{cases}
$$
Thus if $\hat f^n(h_0)\neq h_0$, then $\ol{d}_z(G_\lambda)<\mu^-(d^n)$; if $\hat f^n(h_0)=h_0$, then $\ol{d}_z(G_\lambda)\le\mu^-(d^n)$.
In case (2),
 $$\ol{d}_z(G_\lambda) = \ol{s}_\lambda^n(\vec{w}) < \mu^-(d^n).$$
The induced map $\widehat G_\lambda$ is $[1:0]\in\P^1$ if $\hat f^n(h_0)\neq h_0$; the map $\widehat G_\lambda$ has degree $1$ if $\hat f^n(h_0)=h_0$. Hence by Proposition \ref{stability-proportional-depth}, we have that $G_\lambda$ is stable for all $\lambda \in \Lambda$.

Also by considering the list of cross ratios of the holes of $G_\lambda$, we know $[G_\lambda]$ is not a constant. Therefore, Theorem \ref{Thm:perturbation} holds and $[f]\in I(\Phi_n)$. \hfill $\Box$

\subsection{Polynomial induced map}
\label{polynomial}
In this section, we deal with the maps in Case 4 of Proposition \ref{cases} and prove
\begin{proposition}\label{poly-ind}
Let $d \ge 3$ and consider $f\in\cU_n$ such that $d_{\mathtt{h}}(f)=1$, where  $\mathtt{h}$ is the bad hole. If $\hat{f}(\mathtt{h}) = \mathtt{h}$, then $[f] \in I(\Phi_n)$. Additionally assume that $d\ge 4$. Then Theorem \ref{Thm:perturbation} holds.
\end{proposition}
By Proposition~\ref{depth1poly},  the induced map $\hat{f}$ is conjugate to a  polynomial of degree $d-1$
where $\mathtt{h}$ corresponds to $\infty$. That is, after a change of coordinates
$$\hat{f}(z) = z^{d-1} + a_{d-2} z^{d-2} + \cdots + a_0,$$
for some $a_j \in \C$.

To prove Proposition \ref{poly-ind}, we study two cases: $d \ge 4$ and $d=3$. In the former case, we construct a degenerate holomorphic family $g_{\lambda, t}$ of rational maps such that $g_{\lambda, t}$ converges to $f$ as $t\to 0$, but the limit of $[g^n_{\lambda,t}]$ varies with $\lambda$. To produce such a family $g_{\lambda, t}$, we first perturb the polynomial $\hat f$ by adding two finite poles to  $\hat f$ (decreasing by two the multiplicity of infinity)  and obtain a rational map $\phi_{\lambda,t}$ of the same degree as  $\hat f$.  Then we construct the family $g_{\lambda, t}$ by adding a zero and a pole  to $\phi_{\lambda, t}$. The added zero and pole are suitably chosen so that also $\phi_{\lambda,t}$ and $g_{\lambda,t}$ agree in an appropriate subset of $\poneberk$.  We will show that there exists $\zeta_0\in(\xi_g,\infty)$ such that the coefficient reduction of $g^n_{\lambda, t}$ is stable and varies with $\lambda$. The later case ($d=3$) is more subtle. Since $\deg\hat f=2$,  we were unable to construct appropriate families depending on $\lambda$. In this case we construct two holomorphic families of rational maps with distinct (semi)stable limits under the iteration map. One family is the  analog of the one from the case $d \ge 4$. The other family is constructed by adding a hole and a pole directly to the map $\hat f$.

To ease notation, in this and subsequent sections, we will set $\chi_\alpha = \xi_{0, |t|^{-\alpha}}$ for $\alpha\in\mathbb{R}$. It parametrizes the arc from $0$ to $\infty$ inside $\poneberk$ by $\alpha$.


\subsubsection{Proof of Proposition \ref{poly-ind}: $d\ge 4$ case.}
Consider for the moment $N >1$ to be adjusted in the sequel. For $\lambda\in\C\setminus\{0,1\}$, consider the degree $d-1$ rational map  
$$\phi_{\lambda,t}(z)=\frac{1}{(1-t^Nz)(1-\lambda t^Nz)}\hat f(z) \in \L(z).$$
Given $\alpha >0$, we have that 
$$\phi_{\lambda,t} (\chi_\alpha) =
\begin{cases}
  \chi_{ (d-1) \alpha} & 0 < \alpha \le N,\\
\chi_{ (d-3) \alpha + 2N} & \alpha > N.
\end{cases}
$$
In standard coordinates where we identify $z \in \P^1$ with the
direction at $\chi_\alpha$ containing $z t^{-\alpha}$, we have
$$T_{\chi_\alpha} \phi_{\lambda,t} (z) =
\begin{cases}
  z^{d-1} & 0 < \alpha <N,\\

   \dfrac{1}{(1-z)(1-\lambda z)} z^{d-1} & \alpha =N,\\

\lambda^{-1} z^{d-3} & \alpha > N.
\end{cases}
$$

For an integer $k$ also to be chosen later such that  $0\le k\le n-1$, we let $$\alpha_0 = \dfrac{N}{(d-1)^k}.$$
Set
\begin{eqnarray*}
  \zeta_0 &=& \chi_{\alpha_0},\\
  \zeta_j &=& \phi^j_{\lambda,t} (\zeta_0) \, \text{ for } j \ge 1.
\end{eqnarray*}
Observe that $\zeta_k = \chi_{N}$.

Now let $\vec{v}_\infty\in T_{\z_n}\poneberk$ be the direction containing $\infty$. Pick $c_1\not=c_2\in\mathbb{L}$ in the direction $\vec{v}_\infty$. We put an extra zero at $c_2$ and 
 an extra pole at $c_1$. That is, we consider
$$g_{\lambda,t}(z)=\left(1+\frac{c_1-c_2}{z-c_1}\right)\phi_{\lambda,t}(z).$$
Then $g_{\lambda,t}\to f$ as $t\to 0$. Choose $N$ large such that Lemma \ref{construction} implies that $\z_j = g^j_{\lambda,t} (\z_0)$ and $T_{\z}g_{\lambda,t}=T_\z\phi_{\lambda,t}$ for all $\z\in [\z_0,\z_{n-1}]$. Also we conclude that the surplus multiplicity of $g_{\lambda,t}$ in the direction at $\chi_\alpha$ containing the Gauss point $\xi_g$ is $0$ for $\alpha \le N$ and
$2$ otherwise. Moreover, the surplus multiplicity in the direction at $\chi_\alpha$ containing $\infty$ is $1$ for all $\chi_\alpha \in ]\xi_g, \zeta_n]$.

To ease notation, set
\begin{eqnarray*}
  \mu&=&(d-1)/d,\\
\theta&=&(d-3)/d.
\end{eqnarray*}
Denote by $\vec{w}_0$ and $\vec{w}_\infty$ the directions at $\zeta_0$ containing $0$ and $\infty$, respectively. For the direction $\vec{w}_0$, from above we have
$$\ol{s}_\lambda(T_{\z_0}g_{\lambda,t}^j(\vec{w}_0))=
\begin{cases}
0 & 0\le j\le k,\\
\dfrac{2}{d} & k+1 \le j,
\end{cases}$$
and
$$\ol{m}_\lambda^j(\vec{w}_0)=
\begin{cases}
\mu^j & 1\le j\le k+1,\\
\mu^{k+1}\theta^{j-k-1} & k+2 \le  j.
\end{cases}$$
Also, for the direction $\vec{w}_\infty$, we have
$$\ol{s}_\lambda(T_{\z_0}g_{\lambda,t}^j(\vec{w}_\infty))=\dfrac{1}{d}\ \ \ \  j\ge 0,$$
and
$$\ol{m}_\lambda^j(\vec{w}_\infty)=
\begin{cases}
\mu^j & 1\le j\le k,\\
\mu^k\theta^{j-k} & k+1 \le j.
\end{cases}$$
Then by Lemma \ref{surplus-iterate}, if $n>k+1$,
$$\ol{s}_\lambda^n (\vec{w}_0) =\frac{2}{d}\mu^{k+1}\sum_{j=0}^{n-k-2}\theta^j=\frac{2}{3}\mu^{k+1}(1-\theta^{n-k-1}),$$
and
$$\ol{s}_\lambda^n (\vec{w}_\infty) =\frac{1}{d}\sum_{j=0}^{k}\mu^j+\frac{1}{d}\mu^k\theta\sum_{j=0}^{n-k-2}\theta^j=1-\frac{2}{3}\mu^k-\frac{1}{3}\mu^k\theta^{n-k}.$$
If $n=k+1$, then
$$\ol{s}_\lambda^n (\vec{w}_0) =0,$$
and
$$\ol{s}_\lambda^n (\vec{w}_\infty) =\dfrac{1}{d}\sum_{j=0}^{n-1}\mu^j=1-\mu^n>\mu^+(d^n)$$
since $\hat f^n$ fixes $\infty$. Moreover, for a direction $\vec{w}\in T_{\z_0}\poneberk$, if $T_{\z_0}g^k_{\lambda}(\vec{w})$ contains $t^N$ or $\lambda t^N$, we have
$$\ol{s}_\lambda^n (\vec{w})=\dfrac{1}{3d^{k+1}}(1-\theta^{n-k-1}).$$

Note that as $k$ decreases, $\ol{s}_\lambda^n (\vec{w}_\infty)$ decreases, and for $k=0$,
$$\ol{s}_\lambda^n (\vec{w}_\infty)=\frac{1}{3}(1-\theta^n)<\mu^+(d^n).$$
Hence there exists $0\le k\le n-2$ such that
$$\frac{2}{3}\mu^{k+1}+\frac{1}{3}\mu^{k+1}\theta^{n-k-1}\le\mu^-(d^n)\le\mu^+(d^n)<\frac{2}{3}\mu^k+\frac{1}{3}\mu^k\theta^{n-k}.$$
For such $k$, let $M_t(z) = t^{-\alpha_0} z$, and consider
$$G_{\lambda} = \lim_{t \to 0} M_t^{-1} \circ g^n_{\lambda,t} \circ M_t.$$

By Corollary \ref{relative-position},
$$\ol d_0(G_{\lambda})=\ol{s}^n (\vec{w}_0)+\mu^{k+1}\theta^{n-k-1}=\frac{2}{3}\mu^{k+1}+\frac{1}{3}\mu^{k+1}\theta^{n-k-1}\le\mu^-(d^n),$$
and
$$\ol d_\infty(G_{\lambda})=\ol{s}_\lambda^n (\vec{w}_\infty)<\mu^-(d^n).$$
Moreover, if $z$ is a $(d-1)^k$-th root of unity or a  $(d-1)^k$-th root of $\lambda$, then
$$\ol d_z(G_{\lambda})=\frac{1}{3d^{k+1}}(1-\theta^{n-k-1})<\mu^-(d^n).$$
Note the induced map $\widehat G_\lambda=[1:0]$. Hence by Proposition \ref{stability-proportional-depth}, we have that $G_{\lambda}$ is stable. Moreover, $[G_\lambda] \in \overline{\rat}_{d^n}$ varies with $\lambda$ since $G_\lambda$ has holes at $0, \infty$, the $(d-1)^k$-th roots of $1$ and the $(d-1)^k$-th roots of $\lambda$. Note that the action of $\phi_{\lambda, t}$ and hence of $g_{\lambda, t}$ on the segment $[\xi_g,\infty]$ is independent of $\lambda$. Therefore, the construction is such that Theorem \ref{Thm:perturbation} holds and hence $[f]\in I(\Phi_n)$. \hfill $\Box$

\subsubsection{Proof of Proposition \ref{poly-ind}: cubic case.}\label{cubic-poly} In this case, $\deg\hat f=2$ and $n \ge 2$. As before, we consider
$$\phi_{t}(z)=\frac{1}{1 - t^Nz}\hat f(z),$$
where $N >1$ is to be adjusted in the sequel.

Given $\alpha >0$, note that 
$$\phi_{t} (\chi_\alpha) =
\begin{cases}
  \chi_{2\alpha} & 0 < \alpha \le N,\\

   \chi_{\alpha +N} & \alpha > N.
\end{cases}
$$
In standard coordinates where we identify $z \in \P^1$ with the
direction at $\chi_\alpha$ containing $z t^{-\alpha}$, we have
$$T_{\chi_\alpha} \phi_{t} (z) =
\begin{cases}
  z^2 & \alpha <N,\\

   \dfrac{z^2}{1 - z} & \alpha =N,\\

   - z & \alpha > N.
\end{cases}
$$

Now consider $\zeta_j$ for $j \ge 0$ and $g_t(z)$ as in the proof of previous case. It follows that
$$\ol{s}(T_{\z_0}g_{t}^j(\vec{w}_0))=
\begin{cases}
0 & 0\le j\le k,\\
\dfrac{1}{3} & k+1 \le j \le n.
\end{cases}$$
Choose the largest $k$ such that $0\le k\le n-2$ and
$$2^{k-1}  \left(\frac{1}{3^{k}} + \frac{1}{3^n}\right) > \mu^-(3^n).$$
For such $k$, let $M_t (z) = t^{-N/2^k} z$ and consider
$$G = \lim_{t \to 0} M_t^{-1} \circ g^n_{t} \circ M_t.$$
If $z$ is a $2^k$-th root of unity, then
$$\ol d_z(G)= \dfrac{1}{2} \left(\frac{1}{3^{k+1}} - \frac{1}{3^n}\right).$$
A direct computation shows that $G$ is semistable since
$$\ol{d}_0(G)=2^{k}  \left(\frac{1}{3^{k+1}} + \frac{1}{3^n}\right).$$

Now we consider
$$\psi_{t} (z) =\dfrac{t^Sz -1}{t^S z -(1+t)}\hat{f}(z),$$
where $S >1$ is to be adjusted in the sequel.
Observe that for all  $\alpha>0$,  we have
$$\psi_t (\chi_\alpha) = \chi_{2 \alpha}$$ and
$$T_{\chi_\alpha}\psi_{t} (z) = z^{2}.$$
For $\tilde{k}$ also to be chosen later such that 
$1\le \tilde k < n$, let $$\tilde{\alpha}_0 = \dfrac{S}{2^{\tilde k}}$$
and $$\tilde\zeta_0 = \chi_{\tilde{\alpha}_0}.$$
For $j \ge 1$, let $$\tilde\zeta_j = \psi^j_{t} (\tilde\zeta_0)$$
and observe that $\tilde\zeta_{\tilde k} =\chi_S$.

Denote by $\vec{v}_0$ and $\vec{v}_\infty$ the directions at $T_{\tilde\zeta_0} \poneberk$ containing $0$ and $\infty$, respectively. Then
\begin{eqnarray*}
  \ol{s}^n (\vec{v}_0) &=& \left( \dfrac{2}{3} \right)^{\tilde k+1} -  \left( \dfrac{2}{3} \right)^{n},\\
 \ol{s}^n (\vec{v}_\infty) &=&1-\left( \dfrac{2}{3} \right)^{\tilde k}.
\end{eqnarray*}
Consider one of the $2^{\tilde k}$ directions $\vec{v}$ at $\tilde\zeta_0$ which under
$T_{\tilde\zeta_0} \psi_{t}^{\tilde k}$ maps onto
$\vec{v}_1$ (the direction containing $t^{-S}$).
Then
$$\ol{s}^n (\vec{v})  =   \frac{1}{3^{\tilde k+1}}.$$

Let $L_t (z) = t^{-\tilde\alpha_0} z$ and consider
$$F= \lim_{t \to 0} L_t^{-1} \circ\psi^n_{t} \circ L_t.$$
Then the holes of $F$ are at $0, \infty$ and the $2^{\tilde k}$th-roots of unity. Moreover,
the proportional depths are
$$\ol{d}_0(F) = \ol{s}^n (\vec{v}_0)  + \left( \dfrac{2}{3} \right)^n,$$
$$\ol{d}_\infty(F) = \ol{s}^n (\vec{v}_\infty).$$
If $z$ is a $2^{\tilde k}$-root of unity, then
$$\ol{d}_z(F)= \frac{1}{3^{\tilde k+1}}.$$

Since $\mu^+(3^n)\le\ol{d}_\infty(f^n)=1-(2/3)^n$ there exists $\tilde k <n$ such that $(2/3)^{\tilde k+1} \le\mu^-(3^n) <\mu^+(3^n)<(2/3)^{\tilde k}$.
The induced map is $\widehat F=[1:0]$ and it follows that $F$ is stable for this value of $\tilde k$.

Now we claim that $G$ and $F$ are not in the same GIT-class. For otherwise, $G$ is stable with $k=\tilde k$ and $d_{z}(G)=d_z(F)$ if $z$ is a $2^{k}$th-root of unity. However, if $k=\tilde k$, we  have
$$d_{z}(G)=\frac{1}{2\cdot 3^{k+1}}\left(1-\frac{1}{3^{n-k-1}}\right)< \frac{1}{3^{k+1}}=d_z(F),$$
which is a contradiction. \hfill $\Box$

\begin{corollary}
\label{dimension}
The indeterminacy locus
  $I(\Phi_n)$ contains a complex subvariety of dimension $d-2$ for $n$ sufficiently large.
\end{corollary}

\begin{proof}
  The map $\C^{d-2} \to \ol{\rat}_d$ given by
$$(a_0, \dots, a_{d-3}) \mapsto [a_0 + \cdots + a_{d-3} z^{d-3} + z^{d-1}]$$
is finite--to--one. For $n$ such that $(1 - 1/d)^n < 1/2$, the image of
this map is contained in $\cU_n$.
By the previous
proposition, this image lies in $I(\Phi_n)$.
\end{proof}

\subsection{Monomial induced map}
\label{monomial}
In this section, we deal with the maps in Case 5 of Proposition \ref{cases} and prove
\begin{proposition}\label{monomial-ind}
Let $d \ge 3$ and consider $f\in\cU_n$ such that $d_{\mathtt{h}}(f)=1$, where $\mathtt{h}$ is the bad hole of $f$. If $\hat{f}(\mathtt{h}) \neq \mathtt{h}$, then $[f] \in I(\Phi_n)$. Additionally assume that $d\ge 4$. Then Theorem \ref{Thm:perturbation} holds.
\end{proposition}
In view of Proposition~\ref{depth1poly},  modulo conjugacy,  $\hat{f} (z) =z^{-(d-1)}$ and $\mathtt{h} = \infty$.
We write $n=2m$ or $2m - 1$ for some $m \ge 1$.
Then
$$\ol{d}_\infty (f^n) = \dfrac{1}{d} \sum_{j = 0}^{m-1} \left(\dfrac{d-1}{d} \right)^{2j}.$$
Since $f \in\cU_n$,  we have that $\ol{d}_\infty (f^n ) \ge 1/2$ with strict inequality  if $n$ is odd. Hence $n\ge 3$.

The strategy to prove the proposition is similar to the previous case (polynomial induced map) with the extra complication that the bad hole now has period $2$. 

\subsubsection{Proof of Proposition \ref{monomial-ind}: $d\ge 4$ case.}
Consider for the moment $N >1$ to be adjusted in the sequel. For $\lambda\in\C\setminus\{0, 1\}$, let
$$\phi_{\lambda,t}(z)=(1-t^Nz)(1-\lambda t^Nz)\hat f(z).$$
Since $\hat{f} (z) =z^{-(d-1)}$, the map  $\phi_{\lambda,t}$ also  has degree $d-1$. 

For $\alpha \in \R$, we have 
$$ \phi_{\lambda,t} (\chi_\alpha) =
\begin{cases}
  \chi_{-(d-1)\alpha} & \alpha \le N,\\
 \chi_{-(d-3)\alpha -2N} & \alpha > N.
\end{cases}
$$
In standard coordinates where we identify $z \in \P^1$ with the
direction at $\zeta$ containing $z t^{-\alpha}$, we have
$$T_{\chi_\alpha} \phi_{\lambda,t} (z) =
\begin{cases}
  z^{-(d-1)} & \alpha <N,\\
  (1-z)(1-\lambda z) z^{-(d-1)}& \alpha =N,\\
  \lambda z^{-(d-3)} & \alpha > N.
\end{cases}
$$
As in the proof of the previous proposition,
for an even integer $k$ to be chosen later such that  $0\le k\le n-1$, we let $$\alpha_0 = \dfrac{N}{(d-1)^k}$$
and set
\begin{eqnarray*}
  \zeta_0 &=& \chi_{\alpha_0},\\
  \zeta_j &=& \phi^j_{\lambda,t} (\zeta_0), \, \text{ for } j \ge 1.
\end{eqnarray*}
Observe that $\zeta_k = \chi_N$.

Now let $\vec{v}_\infty\in T_{\z_{2m}}\poneberk$ be the direction containing $\infty$. Pick $c_1\not=c_2\in\mathbb{L}$ in the direction $\vec{v}_\infty$. Consider
$$g_{\lambda,t}(z)=\left(1+\frac{c_1-c_2}{z-c_1}\right)\phi_{\lambda,t}(z).$$
Then $g_{\lambda,t}\to f$ as $t\to 0$. Choose $N$ large such that Lemma \ref{construction} implies that $\z_j = g^j_{\lambda,t} (\z_0)$ and $T_{\z}g_{\lambda,t}=T_\z\phi_{\lambda,t}$ 
for $\z\in [\z_{2m-1},\z_{2m-2}]$. Moreover, a direction at such $\z$ has positive surplus multiplicity if and only if it contains $\infty$. In this case, the surplus multiplicity is $1$. \par

Again we set $\mu=(d-1)/d$ and $\theta=(d-3)/d$.
 Denote by $\vec{w}_0$ and $\vec{w}_\infty$ the directions at $T_{\z_0} \poneberk$ containing $0$ and $\infty$, respectively.
For the direction $\vec{w}_\infty$, we have
$$\ol{s}_\lambda(T_{\z_0}g_{\lambda,t}^j(\vec{w}_\infty))=
\begin{cases}
\dfrac{1}{d} &\text{even}\ j\ge 0,\\
0 &\text{odd}\ j\ge 1,
\end{cases}$$
and
$$\ol{m}_\lambda^j(\vec{w}_\infty)=
\begin{cases}
\mu^j & 1\le j\le k,\\
\mu^k\theta^{(j-k+1)/2}\mu^{(j-k-1)/2} & \text{odd}\ j\ge k+1,\\
\mu^k\theta^{(j-k)/2}\mu^{(j-k)/2} & \text{even}\ j\ge k+2.
\end{cases}$$
Then by Lemma \ref{surplus-iterate},
$$\ol{s}_\lambda^n (\vec{w}_\infty) =\frac{1}{d}\sum_{j=0}^{k/2}\mu^{2j}+\frac{1}{d}\mu^k\sum_{j=1}^{m-k/2-1}\mu^j\theta^{j}.$$
Let $\vec{w}\in T_{\z_0} \poneberk$ be a direction such that $T_{\z_0}g_{\lambda,t}^k(\vec{w})$ is a direction containing $t^N$ or $\lambda t^N$. Then there are $2(d-1)^k$ many such $\vec{w}$, and under $g_{\lambda,t}^n$ each of them has proportional surplus multiplicity
$$\ol s^n_\lambda(\vec{w})=\frac{1}{d}\cdot\frac{1}{d^{k+1}}\sum_{j=1}^{m-k/2-1}\mu^j\theta^{j-1}.$$
Thus
$$\ol s^n_\lambda(\vec{w}_\infty)+2(d-1)^k\ol s^n_\lambda(\vec{v})=\frac{1}{d}\sum_{j=0}^{k/2}\mu^{2j}+\frac{1}{d}\mu^{k+2}\sum_{j=0}^{m-k/2-2}\mu^j\theta^{j}=\frac{1}{d}\sum_{j=0}^{k/2+1}\mu^{2j}+\frac{1}{d}\mu^{k+2}\sum_{j=1}^{m-k/2-2}\mu^j\theta^{j}.$$

Set
$$a_k=\frac{1}{d}\sum_{j=0}^{k/2}\mu^{2j}+\frac{1}{d}\mu^k\sum_{j=1}^{m-k/2-1}\mu^j\theta^{j}.$$
Then $a_k$ decreases as even $k$ decreases. Note that
$a_{2m-2}=\ol{d}_\infty (f^n)$ and $a_{0}<\mu^-(d^n)\le \mu^+(d^n)$. Thus there exists even integer $0\le k\le 2m-2$ such that
$$\begin{cases}
a_k<\mu^-(d^n)\le\mu^+(d^n)\le a_{k+2} &\text{if}\ n \text{ is even},\\
a_k\le\mu^-(d^n)\le\mu^+(d^n)<a_{k+2} &\text{if}\ n \text{ is odd}.
\end{cases}$$
For this $k$, we have
$$\begin{cases}
\ol s^n_\lambda(\vec{w}_\infty)<\mu^-(d^n)\ \text{and}\ \sum\limits_{\vec{v}\not=\vec{w}_0}\ol s^n_\lambda(\vec{v})\ge \mu^+(d^n) &\text{if}\ n \text{ is even},\\
\ol s^n_\lambda(\vec{w}_\infty)\le\mu^-(d^n)\ \text{and}\ \sum\limits_{\vec{v}\not=\vec{w}_0}\ol s^n_\lambda(\vec{v})> \mu^+(d^n) &\text{if}\ n \text{ is odd}.
\end{cases}$$

For such $k$, let $M_t(z) = t^{-\alpha_0} z$ and consider
$$G_{\lambda} = \lim_{t \to 0} M_t^{-1} \circ g^n_{\lambda,t} \circ M_t.$$
Since $g^n_{\lambda,t}(\z_0)\not=\z_0$, the induced map $\widehat G_{\lambda}$ is a constant. Hence
$$\sum_{z\in\P^1}\ol d_z(G_{\lambda})=1.$$
By Lemma \ref{relative-position}, if $\vec{v}_z\not=\vec{w}_0\in T_{\z_0}\poneberk$, we have $\ol d_z(G_{\lambda})=\ol s_\lambda^n(\vec{v}_z)$.
Thus
$$\ol d_\infty(G_{\lambda})=\ol{s}_\lambda^n (\vec{w}_\infty)$$
and
$$\ol d_0(G_{\lambda})=1-\sum_{\vec{v}\not=\vec{w}_0}\ol s_\lambda^n(\vec{v}).$$
It follows that
$$\begin{cases}
\ol d_\infty(G_{\lambda})<\mu^-(d^n)\ \text{and}\ \ol d_0(G_{\lambda})\le \mu^-(d^n) &\text{if}\ n \text{ is even},\\
\ol d_\infty(G_{\lambda})\le\mu^-(d^n)\ \text{and}\ \ol d_0(G_{\lambda})<\mu^-(d^n) &\text{if}\ n \text{ is odd}.
\end{cases}$$

Moreover, if $z$ is a $(d-1)^k$-th root of unity or a  $(d-1)^k$-th root of $\lambda$, then
$$\ol d_z(G_{\lambda})=\frac{1}{d}\cdot\frac{1}{d^{k+1}}\sum_{j=1}^{m-k/2-1}\mu^j\theta^{j-1}<\mu^-(d^n).$$
Note that if $n$ is even, the induced map $\widehat G_{\lambda}$ is $[1:0]\in\P^1$; if $n$ is odd, $\widehat G_{\lambda}$ is $[0:1]\in\P^1$. Hence by Proposition~\ref{stability-proportional-depth}, $G_{\lambda}$ is stable. Moreover, $[G_\lambda]$ varies with $\lambda$. Note that the action of $\phi_{\lambda, t}$ and hence of $g_{\lambda, t}$ on the segment $[0,\infty]$ is independent of $\lambda$. Therefore, the construction is such that Theorem \ref{Thm:perturbation} holds and hence $[f]\in I(\Phi_n)$. \hfill $\Box$

\subsubsection{Proof of Proposition \ref{monomial-ind}: cubic case.}\label{cubic-monomial}
In this case, $\deg\hat f=2$. First as in the previous case, consider
$$\phi_{t}(z)=(1-t^Nz)\hat f(z),$$
where $N >1$ is sufficiently large.

Given $\alpha\in\mathbb{R}$, note that 
$$\phi_{t} (\chi_\alpha) =
\begin{cases}
  \chi_{-2\alpha} &  \alpha \le N,\\
  \chi_{-\alpha -N} & \alpha > N.
\end{cases}
$$
In standard coordinates where we identify $z \in \P^1$ with the
direction at $\zeta$ containing $z t^{-\alpha}$, we have
$$T_{\chi_\alpha} \phi_{t} (z) =
\begin{cases}
  z^{-2} & \alpha <N,\\

  (1-z)z^{-2}& \alpha =N,\\

- z^{-1} & \alpha > N.
\end{cases}
$$

Now consider $\zeta_j$ for $j \ge 0$ and $g_t(z)$ as in the proof of previous case.
Set
$$a_k=\frac{1}{3}\sum_{j=0}^{k/2}\left(\frac{2}{3}\right)^{2j}+\frac{1}{3}\left(\frac{2}{3}\right)^k\sum_{j=1}^{m-k/2-1}\left(\frac{2}{3}\right)^j\left(\frac{1}{3}\right)^{j}.$$
Choose an even integer $0\le k\le 2m-2$ such that
$$\begin{cases}
a_k<\mu^-(3^n)\le\mu^+(3^n)\le a_{k+2} &\text{if}\ n \text{ is even},\\
a_k\le\mu^-(3^n)\le\mu^+(3^n)<a_{k+2} &\text{if}\ n \text{ is odd}.
\end{cases}$$
For such $k$, let $M_t(z) = t^{-N/2^k} z$
and consider
$$G = \lim_{t \to 0} M_t^{-1} \circ g^n_{t} \circ M_t.$$ 
If $z$ is a $2^k$-th root of unity, then
$$\ol d_z(G)=\frac{1}{3^{k+2}}\sum_{j=1}^{m-k/2-1}\left(\frac{2}{3}\right)^{j}\left(\frac{1}{3}\right)^{j-1},$$
A direct computation shows that $G$ is stable since
$$\ol d_\infty(G)=a_k,$$
$$\ol d_0(G)=1-\sum_{z\not=0}\ol d_z(G)=1-a_{k+2},$$
and
if $n$ is even, $\widehat G=[1:0]\in\P^1$; if $n$ is odd, $\widehat G=[0:1]\in\P^1$.

Now we consider
$$\psi_{t} (z) =\dfrac{t^Sz -1}{t^S z -(1+t)}\hat{f}(z),$$
where $S >1$ is sufficiently large.
Observe that given $\alpha \in \R$, we have $\psi_t (\chi_\alpha) = \chi_{2 \alpha}$ and
$$T_{\chi_\alpha}\psi_{t} (z) = z^{-2}.$$
For an even integer $\tilde k$ to be chosen later with $1 \le \tilde k < n$, let $$\tilde{\alpha}_0 = \dfrac{S}{2^{\tilde k}}$$
and  $$\tilde\zeta_0 =  \chi_{\tilde{\alpha}_0}.$$
For $j \ge 1$, let $$\tilde\zeta_j = \psi^j_{t} (\tilde\zeta_0)$$
and observe that $\tilde\zeta_{\tilde k} = \chi_{S}$.

Let $\vec{v}_0$ (resp. $\vec{v}_\infty$) be the direction in $T_{\tilde\zeta_0}\poneberk$ containing $0$ (resp. $\infty$). Observe that $\tilde \zeta_{\tilde k}$ is contained in the direction $T_{\tilde\zeta_0}\psi^j_{t} (\vec{v}_0)$ or  in $T_{\tilde\zeta_0}\psi^j_{t} (\vec{v}_\infty)$ for all $j \neq\tilde k$.
Thus,
$$\ol{s}^n (\vec{v}_0) + \ol{s}^n (\vec{v}_\infty) +  \dfrac{1}{3} \left(\dfrac{2}{3} \right)^{\tilde k} = \sum_{j=0}^{n-1} \dfrac{1}{3}\left(\dfrac{2}{3} \right)^{j}= 1 - \left(\dfrac{2}{3} \right)^{n}.$$
Since $T_{\tilde\zeta_0}\psi^j_{t} (\vec{v}_\infty)$ contains  $\tilde\zeta_{\tilde k}$ only for $j$ even with $j <\tilde k$, we have
$$\ol{s}^n (\vec{v}_\infty) = \dfrac{1}{3} \sum_{j=0}^{\tilde k/2 -1} \left(\dfrac{2}{3} \right)^{2j}.$$

Choose $\tilde k$ even such that if $n$ is even,
$$\dfrac{1}{3} \sum_{j=0}^{\tilde k/2 -1} \left(\dfrac{2}{3} \right)^{2j} <\mu^-(3^n)<\mu^+(3^n)\le
\dfrac{1}{3} \sum_{j=0}^{\tilde k/2} \left(\dfrac{2}{3} \right)^{2j},$$
and if $n$ is odd,
$$\dfrac{1}{3} \sum_{j=0}^{\tilde k/2 -1} \left(\dfrac{2}{3} \right)^{2j}\le\mu^-(3^n)<\mu^+(3^n)<
\dfrac{1}{3} \sum_{j=0}^{\tilde k/2} \left(\dfrac{2}{3} \right)^{2j}.$$

Thus, if $n$ is even,
$$ \ol{s}^n (\vec{v}_\infty) <  \mu^-(3^n)\ \text{and}\ \ol{s}^n (\vec{v}_0) \le  \mu^-(3^n) -  \left(\dfrac{2}{3} \right)^{n},$$
and
if $n$ is odd,
$$ \ol{s}^n (\vec{v}_\infty) \le \mu^-(3^n)\ \text{and}\ \ol{s}^n (\vec{v}_0) < \mu^-(3^n) -  \left(\dfrac{2}{3} \right)^{n}.$$
If $\vec{v}$ is one of the $2^{\tilde k}$ directions at $\tilde \zeta_0$ for which $T_{\tilde \zeta_0}\psi^{\tilde k}_{t} (\vec{v})$
contains $t^{-S}$, then
$$\ol{s}^n(\vec{v}) =\frac{1}{3^{\tilde k+1}}.$$

Now let $L_t(z) = t^{-\tilde\alpha_0 z}$ and
$$F= \lim_{t \to 0}  L_t^{-1} \circ\psi^n_{t} \circ L_t \in \ol{\Rat}_{d^n}.$$
The holes of $F$ are $0$, $\infty$, and all $z$ such that $T_{\tilde \zeta_0}\psi^{\tilde k}_{t} (z)=1$. Moreover, we have
$$\ol{d}_0(F)=\ol{s}^n (\vec{v}_0)+\left(\dfrac{2}{3} \right)^{n}.$$
If $\vec{v}_z$ is one of the $2^{\tilde k}$ directions at $\tilde \zeta_0$ for which $T_{\tilde \zeta_0}\psi^{\tilde k}_{t} (\vec{v}_z)$
contains $t^{-S}$, then
$$\ol{d}_z(F)=\ol{s}^n (\vec{v}_z) =\frac{1}{3^{\tilde k+1}}<\mu^-(3^n).$$
Note that if $n$ is even, $\widehat F=[1:0]$, and if $n$ is odd, $\widehat F=[0:1]$. It follows that $F$ is stable.

Now we claim that $G$ and $F$ are not in the same GIT-class. For otherwise,  $k=\tilde k$ and $d_{z}(G)=d_z(F)$ if $z$ is a  $2^{k}$th-root  of unity. However, if $k=\tilde k$, we have
$$d_{z}(G)=\frac{1}{3^{k+2}}\sum_{j=1}^{m-k/2-1}\frac{1}{3^{2j-1}}< \frac{1}{3^{k+1}}=d_z(F),$$
which is a contradiction. \hfill $\Box$

\section{Constant induced map}
\label{constant}

In this section, we deal with the constant induced map case and complete the proof of Theorem \ref{main}. More precisely, we prove
\begin{proposition}
\label{constant-p}
If $f\in I(d)\cup\cU_n$ is semistable and has constant induced map $\hat f$. Then $[f]\in I(\Phi_n)$.
\end{proposition}

First we establish some general properties of degenerate rational maps
 with constant induced  maps. Then in Section~\ref{proof-constant}
we prove the above proposition.

\subsection{Stable and $n$-unstable maps with constant induced map}

Recall that
$$I(d)=\{f=H_f\hat f\in\P^{2d+1}: \hat f=c\in\P^1, H_f(c)=0\}. $$
By Proposition \ref{stability-depth}, we immediately have
\begin{lemma}\label{deg4}
For $d\ge 2$, then
$\mathrm{Rat}_d^s\cap I(d)=\emptyset$ if and only if $d=2$ or $3$.
\end{lemma}

For degrees $d\ge 4$, we have a
lower bound on the number of holes of the maps in $\mathrm{Rat}_d^s\cap I(d)$:

\begin{lemma}\label{num-hole}
For $d\ge 4$, if $f\in\mathrm{Rat}_d^s\cap I(d)$, then $f$ has at least $3$ holes.
\end{lemma}
\begin{proof}
We proceed by contradiction. Suppose $f$ has two holes, $h_1, h_2$ with depths $d_1,d_2$. Write $f=H_f\hat f$. Then $\deg H_f=d$ and $d_1+d_2=d$. Since $f\in\mathrm{Rat}_d^s$, by 
Proposition~\ref{stability-depth}, $d_1\le d/2$ and $d_2\le d/2$. Hence $d_1=d_2=d/2$. Since $f\in I(d)$, the induced map $\hat f$ is a constant which is a hole, say $h_1$. Again by Proposition~\ref{stability-depth}, we have $d_1< d/2$, since $f\in \mathrm{Rat}_d^s$, which is a contradiction.
\end{proof}

For the $n$-unstable set we have

\begin{lemma}\label{deg0}
Consider $f \in\cU_n$ with bad hole $\mathtt{h}$.
If $\deg\hat f=0$, then $d_{\mathtt{h}}(f)=(d+1)/2$.
\end{lemma}
\begin{proof}
It is a direct consequence of  Lemma~\ref{depth-multiplicity-inequality}
since $\deg\hat f=0$ and therefore $m_{\mathtt{h}} (\hat f ) =0$.
\end{proof}
 If $f \in I(d)$ is strictly semistable or $f$ is $n$-unstable with constant induced map, then $f$ is GIT conjugate to a degenerate map with non-constant induced map:
\begin{lemma}\label{const-to-nonconst}
For odd $d\ge 3$, set
$$F([X:Y])=X^{\frac{d+1}{2}}Y^{\frac{d-1}{2}}[1:0]\in\P^{2d+1}$$
and
$$G([X:Y])=X^{\frac{d-3}{2}}Y^{\frac{d-1}{2}}[X^2:Y^2]\in\P^{2d+1}.$$
Then
\begin{enumerate}
\item $[F]=[G]$.
\item If $f\in(\mathrm{Rat}_d^{ss}\setminus\mathrm{Rat}_d^s)\cap I(d)$ or $f\in\cU_n$ with constant induced map $\hat f$, then $[f]=[F]$.
\end{enumerate}
\end{lemma}
\begin{proof}
Let $M_t(z)=tz$. Then as $t\to 0$,
$M_t\circ G\circ M_t^{-1}\to F$.
Since $F$ and $G$ are in $\mathrm{Rat}_d^{ss}\setminus\mathrm{Rat}_d^s$, we have $[F]=[G]$.

When $f \in \cU_n$ has constant induced map, by Lemma ~\ref{deg0}, it follows that $f$ is M\"obius conjugate to 
$$f_{-1}([X:Y])=X^{\frac{d+1}{2}}H_{-1}(X,Y)[1:0],$$
where $H_{-1}(X,Y)$ is a degree $(d-1)/2$ homogenous polynomial and $H_{-1}([0:1])\not=0.$
When $f\in(\mathrm{Rat}_d^{ss}\setminus\mathrm{Rat}_d^s)\cap I(d)$, by Proposition ~\ref{stability-depth} we have that $f$ is M\"obius conjugate to
$$f_1([X:Y])=X^{\frac{d-1}{2}}H_1(X,Y)[0:1], $$
where $H_1(X,Y)$ is a degree $(d+1)/2$ homogenous polynomial and $H_1([0:1])\not=0$.

For $i\in\{-1, 1\}$, set $M_{t,i}(z)=t^iz^{-i}$. Then as $t\to 0$,
$$M_{t,i}\circ f_i\circ M_{t,i}^{-1}\to F.$$
Note $f_{-1},f_1$ and $F$ are semistable. We have $[f]=[F]$.
\end{proof}

\subsection{Proof of Proposition~\ref{constant-p}}
\label{proof-constant}
If $f$ satisfies the assumptions in Lemma \ref{const-to-nonconst} (2), then $[f]=[G]$, where $G$ is defined as in Lemma \ref{const-to-nonconst}. We can apply either Proposition \ref{depth>1} or \ref{poly-ind} to $G$ depending whether $d\ge 4$ or $d=3$ to conclude that $[f] \in I(\Phi_n)$. So it is sufficient to assume $f\in\mathrm{Rat}_d^s\cap I(d)$.
By Lemmas \ref{deg4} and \ref{num-hole}, we may assume $d\ge 4$ and normalize such  that $f$ has holes at $0,1$ and $\infty$. Thus $f$ has the form:
$$f([X:Y])=H_f(X,Y)\hat f([X:Y])=X^{d_0}Y^{d_\infty}(X-Y)^{d_1}\prod_{i=2}^k(X-c_iY)^{d_i}[0:1],$$
where $1\le d_1\le d/2$, $1\le d_i<d/2$ for $i\in\{0,2,\cdots,k,\infty\}$ and $c_2,\cdots, c_k$ are distinct points in $\mathbb{C}\setminus\{0,1\}$. \par
For $t\in\mathbb{C}\setminus\{0\}$, set
$$g_t([X:Y])=H_{f}(X,Y)[t:1]$$
and
$$h_t([X:Y])=\frac{H_{f}(X,Y)}{Y}[tX:Y].$$
Then $g_t$ and $h_t$ are stable but not in $I(d)$ for sufficiently small $t\not=0$. Moreover, $g_t$ and $h_t$ converges to $f$ as $t\to 0$. Hence $[g_t]$ and $[h_t]$ converge to $[f]$ as $t\to 0$.

Note
$$g_t^n([X:Y])=(H_f(X,Y))^{d^{n-1}}[t:1]$$
and
$$h_t^n([X:Y])=\prod_{m=0}^{n-1}\left(\frac{H_f(t^mX,Y)}{Y}\right)^{d^{n-1-m}}[t^nX:Y].$$
Then for sufficiently small $t\not=0$, by Lemma \ref{stability-depth}, $g_t^n$ and $h_t^n$ are stable. 
Set $g_n=\lim\limits_{t\to 0}g_t^n$ and $h_n=\lim\limits_{t\to 0}h_t^n$.
Then we have
$$g_n([X:Y])=(H_f(X,Y))^{d^{n-1}}[0:1].$$
If $m \neq 0$, 
then $$H_f (t^m X, Y) = t^{m d_0} X^{d_0}Y^{d_\infty}(t^m X-Y)^{d_1}\prod_{i=2}^k(t^m X-c_iY)^{d_i} \to X^{d_0} Y^{d-d_0}$$
in the projective space of the vector space of homogenous polynomials of degree $d$.
Taking into account that 
$$\prod_{m=1}^{n-1}  \left( X^{ d_0} Y^{(d-d_0-1)} \right)^{d^{n-1-m}} = \left( X^{ d_0} Y^{(d-d_0-1)} \right)^{\frac{d^{n-1}-1}{d-1}}, $$
it follows that
$$h_n([X:Y]) = \left(\frac{H_f(X,Y)}{Y}\right)^{d^{n-1}} \left( X^{ d_0} Y^{(d-d_0-1)} \right)^{\frac{d^{n-1}-1}{d-1}} Y [0:1].$$
After collecting all powers of $X$ and $Y$ we obtain
$$h_n([X:Y])=\left(\frac{H_f(X,Y)}{X^{d_0}Y^{d_\infty}}\right)^{d^{n-1}} X^{\frac{d^n-1}{d-1}d_0}Y^{d^{n-1}d_\infty-\frac{d^{n-1}-1}{d-1}d_0}[0:1].$$
Note $d_0(g_n)<d^n/2$, $d_0(h_n)<d^n/2$ and the depths of all other holes of $g_n$ and $h_n$ are $\le d^n/2$. Therefore $g_n$ and $h_n$ are stable. Thus, as $t\to 0$, $\Phi_n([g_t])$ converges to $[g_n]$ and $\Phi_n([h_t])$ converges to $[h_n]$. However, $[g_n]\neq [h_n]$ since $d_0(g_n)\not=d_0(h_n)$ and the induced map for both $g_n$ and $h_n$ is the constant map $[0:1]$. Thus $[f]\in I(\Phi_n)$ for all $n\ge 2$.
\hfill $\Box$

\bibliographystyle{siam}
\bibliography{references}
\end{document}